\documentclass[a4paper, british]{amsart}

\usepackage{libertine}
\usepackage[libertine]{newtxmath}

\usepackage{amssymb}
\usepackage{babel}
\usepackage{enumitem}
\usepackage{hyperref}
\usepackage[utf8]{inputenc}
\usepackage{newunicodechar}
\usepackage{mathtools}
\usepackage{varioref}
\usepackage[arrow,curve,matrix]{xy}

\usepackage{colortbl}
\usepackage{graphicx}
\usepackage{tikz}

\usepackage{mathrsfs}

\definecolor{linkred}{rgb}{0.7,0.2,0.2}
\definecolor{linkblue}{rgb}{0,0.2,0.6}

\numberwithin{figure}{section}

\usepackage[hyperpageref]{backref}

\setdescription{labelindent=\parindent, leftmargin=2\parindent}
\setitemize[1]{labelindent=\parindent, leftmargin=2\parindent}
\setenumerate[1]{labelindent=0cm, leftmargin=*, widest=iiii}

\newunicodechar{α}{\ensuremath{\alpha}}
\newunicodechar{β}{\ensuremath{\beta}}
\newunicodechar{χ}{\ensuremath{\chi}}
\newunicodechar{δ}{\ensuremath{\delta}}
\newunicodechar{ε}{\ensuremath{\varepsilon}}
\newunicodechar{Δ}{\ensuremath{\Delta}}
\newunicodechar{η}{\ensuremath{\eta}}
\newunicodechar{γ}{\ensuremath{\gamma}}
\newunicodechar{Γ}{\ensuremath{\Gamma}}
\newunicodechar{ι}{\ensuremath{\iota}}
\newunicodechar{κ}{\ensuremath{\kappa}}
\newunicodechar{λ}{\ensuremath{\lambda}}
\newunicodechar{Λ}{\ensuremath{\Lambda}}
\newunicodechar{ν}{\ensuremath{\nu}}
\newunicodechar{μ}{\ensuremath{\mu}}
\newunicodechar{ω}{\ensuremath{\omega}}
\newunicodechar{Ω}{\ensuremath{\Omega}}
\newunicodechar{π}{\ensuremath{\pi}}
\newunicodechar{Π}{\ensuremath{\Pi}}
\newunicodechar{φ}{\ensuremath{\phi}}
\newunicodechar{Φ}{\ensuremath{\Phi}}
\newunicodechar{ψ}{\ensuremath{\psi}}
\newunicodechar{Ψ}{\ensuremath{\Psi}}
\newunicodechar{ρ}{\ensuremath{\rho}}
\newunicodechar{σ}{\ensuremath{\sigma}}
\newunicodechar{Σ}{\ensuremath{\Sigma}}
\newunicodechar{τ}{\ensuremath{\tau}}
\newunicodechar{θ}{\ensuremath{\theta}}
\newunicodechar{Θ}{\ensuremath{\Theta}}
\newunicodechar{ξ}{\ensuremath{\xi}}
\newunicodechar{Ξ}{\ensuremath{\Xi}}
\newunicodechar{ζ}{\ensuremath{\zeta}}

\newunicodechar{ℓ}{\ensuremath{\ell}}

\newunicodechar{𝔹}{\ensuremath{\bB}}
\newunicodechar{ℂ}{\ensuremath{\bC}}
\newunicodechar{𝔻}{\ensuremath{\bD}}
\newunicodechar{𝔼}{\ensuremath{\bE}}
\newunicodechar{𝔽}{\ensuremath{\bF}}
\newunicodechar{ℕ}{\ensuremath{\bN}}
\newunicodechar{ℙ}{\ensuremath{\bP}}
\newunicodechar{ℚ}{\ensuremath{\bQ}}
\newunicodechar{ℝ}{\ensuremath{\bR}}
\newunicodechar{𝕏}{\ensuremath{\bX}}
\newunicodechar{ℤ}{\ensuremath{\bZ}}
\newunicodechar{𝒜}{\ensuremath{\sA}}
\newunicodechar{ℬ}{\ensuremath{\sB}}
\newunicodechar{𝒞}{\ensuremath{\sC}}
\newunicodechar{𝒟}{\ensuremath{\sD}}
\newunicodechar{ℰ}{\ensuremath{\sE}}
\newunicodechar{ℱ}{\ensuremath{\sF}}
\newunicodechar{𝒢}{\ensuremath{\sG}}
\newunicodechar{ℋ}{\ensuremath{\sH}}
\newunicodechar{𝒥}{\ensuremath{\sJ}}
\newunicodechar{ℒ}{\ensuremath{\sL}}
\newunicodechar{𝒪}{\ensuremath{\sO}}
\newunicodechar{𝒬}{\ensuremath{\sQ}}
\newunicodechar{𝒯}{\ensuremath{\sT}}
\newunicodechar{𝒲}{\ensuremath{\sW}}

\newunicodechar{∂}{\ensuremath{\partial}}
\newunicodechar{∇}{\ensuremath{\nabla}}

\newunicodechar{↺}{\ensuremath{\circlearrowleft}}
\newunicodechar{∞}{\ensuremath{\infty}}
\newunicodechar{⊕}{\ensuremath{\oplus}}
\newunicodechar{⊗}{\ensuremath{\otimes}}
\newunicodechar{•}{\ensuremath{\bullet}}
\newunicodechar{Λ}{\ensuremath{\wedge}}
\newunicodechar{↪}{\ensuremath{\into}}
\newunicodechar{→}{\ensuremath{\to}}
\newunicodechar{↦}{\ensuremath{\mapsto}}
\newunicodechar{⨯}{\ensuremath{\times}}
\newunicodechar{∪}{\ensuremath{\cup}}
\newunicodechar{∩}{\ensuremath{\cap}}
\newunicodechar{⊋}{\ensuremath{\supsetneq}}
\newunicodechar{⊇}{\ensuremath{\supseteq}}
\newunicodechar{⊃}{\ensuremath{\supset}}
\newunicodechar{⊊}{\ensuremath{\subsetneq}}
\newunicodechar{⊆}{\ensuremath{\subseteq}}
\newunicodechar{⊂}{\ensuremath{\subset}}
\newunicodechar{≥}{\ensuremath{\geq}}
\newunicodechar{≠}{\ensuremath{\neq}}
\newunicodechar{≫}{\ensuremath{\gg}}
\newunicodechar{≪}{\ensuremath{\ll}}

\newunicodechar{≤}{\ensuremath{\leq}}
\newunicodechar{∈}{\ensuremath{\in}}
\newunicodechar{∉}{\ensuremath{\not \in}}
\newunicodechar{∖}{\ensuremath{\setminus}}
\newunicodechar{◦}{\ensuremath{\circ}}
\newunicodechar{°}{\ensuremath{^\circ}}
\newunicodechar{…}{\ifmmode\mathellipsis\else\textellipsis\fi}
\newunicodechar{·}{\ensuremath{\cdot}}
\newunicodechar{⋯}{\ensuremath{\cdots}}
\newunicodechar{∅}{\ensuremath{\emptyset}}
\newunicodechar{⇒}{\ensuremath{\Rightarrow}}

\newunicodechar{⁰}{\ensuremath{^0}}
\newunicodechar{¹}{\ensuremath{^1}}
\newunicodechar{²}{\ensuremath{^2}}
\newunicodechar{³}{\ensuremath{^3}}
\newunicodechar{⁴}{\ensuremath{^4}}
\newunicodechar{⁵}{\ensuremath{^5}}
\newunicodechar{⁶}{\ensuremath{^6}}
\newunicodechar{⁷}{\ensuremath{^7}}
\newunicodechar{⁸}{\ensuremath{^8}}
\newunicodechar{⁹}{\ensuremath{^9}}
\newunicodechar{ⁱ}{\ensuremath{^i}}

\newunicodechar{⌈}{\ensuremath{\lceil}}
\newunicodechar{⌉}{\ensuremath{\rceil}}
\newunicodechar{⌊}{\ensuremath{\lfloor}}
\newunicodechar{⌋}{\ensuremath{\rfloor}}

\newunicodechar{≅}{\ensuremath{\cong}}
\newunicodechar{⇔}{\ensuremath{\Leftrightarrow}}

\DeclareFontFamily{OMS}{rsfs}{\skewchar\font'60}
\DeclareFontShape{OMS}{rsfs}{m}{n}{<-5>rsfs5 <5-7>rsfs7 <7->rsfs10 }{}
\DeclareSymbolFont{rsfs}{OMS}{rsfs}{m}{n}
\DeclareSymbolFontAlphabet{\scr}{rsfs}
\DeclareSymbolFontAlphabet{\scr}{rsfs}

\DeclareFontFamily{U}{mathx}{\hyphenchar\font45}
\DeclareFontShape{U}{mathx}{m}{n}{
      <5> <6> <7> <8> <9> <10>
      <10.95> <12> <14.4> <17.28> <20.74> <24.88>
      mathx10
      }{}
\DeclareSymbolFont{mathx}{U}{mathx}{m}{n}
\DeclareFontSubstitution{U}{mathx}{m}{n}
\DeclareMathAccent{\wcheck}{0}{mathx}{"71}

\DeclareMathOperator{\Aut}{Aut}

\DeclareMathOperator{\Hom}{Hom}

\DeclareMathOperator{\Pic}{Pic}
\DeclareMathOperator{\rank}{rank}

\DeclareMathOperator{\reg}{reg}

\DeclareMathOperator{\sEnd}{\sE\negthinspace \mathit{nd}}
\DeclareMathOperator{\sing}{sing}

\DeclareMathOperator{\Sym}{Sym}

\DeclareMathOperator{\tor}{tor}

\newcommand{\sA}{\scr{A}}
\newcommand{\sB}{\scr{B}}
\newcommand{\sC}{\scr{C}}
\newcommand{\sD}{\scr{D}}
\newcommand{\sE}{\scr{E}}
\newcommand{\sF}{\scr{F}}
\newcommand{\sG}{\scr{G}}
\newcommand{\sH}{\scr{H}}

\newcommand{\sJ}{\scr{J}}

\newcommand{\sL}{\scr{L}}

\newcommand{\sN}{\scr{N}}
\newcommand{\sO}{\scr{O}}

\newcommand{\sQ}{\scr{Q}}

\newcommand{\sS}{\scr{S}}
\newcommand{\sT}{\scr{T}}

\newcommand{\sW}{\scr{W}}

\newcommand{\cC}{\mathcal C}

\newcommand{\bA}{\mathbb{A}}
\newcommand{\bB}{\mathbb{B}}
\newcommand{\bC}{\mathbb{C}}
\newcommand{\bD}{\mathbb{D}}
\newcommand{\bE}{\mathbb{E}}
\newcommand{\bF}{\mathbb{F}}

\newcommand{\bN}{\mathbb{N}}

\newcommand{\bP}{\mathbb{P}}
\newcommand{\bQ}{\mathbb{Q}}
\newcommand{\bR}{\mathbb{R}}

\newcommand{\bX}{\mathbb{X}}

\newcommand{\bZ}{\mathbb{Z}}

\newcommand{\aB}{{\sf B}}

\newcommand{\aE}{{\sf E}}
\newcommand{\aF}{{\sf F}}

\theoremstyle{plain}
\newtheorem{thm}{Theorem}[section]

\newtheorem{cor}[thm]{Corollary}
\newtheorem{defn}[thm]{Definition}
\newtheorem{fact}[thm]{Fact}
\newtheorem{lem}[thm]{Lemma}

\newtheorem{prop}[thm]{Proposition}

\theoremstyle{remark}

\newtheorem{asswlog}[thm]{Assumption w.l.o.g.}
\newtheorem{claim}[thm]{Claim}
\newtheorem{c-n-d}[thm]{Claim and Definition}

\newtheorem{construction}[thm]{Construction}

\newtheorem{example}[thm]{Example}

\newtheorem{notation}[thm]{Notation}

\newtheorem{rem}[thm]{Remark}
\newtheorem{question}[thm]{Question}
\newtheorem*{rem-nonumber}{Remark}

\numberwithin{equation}{thm}

\setlist[enumerate]{label=(\thethm.\arabic*), before={\setcounter{enumi}{\value{equation}}}, after={\setcounter{equation}{\value{enumi}}}}

\newcommand{\into}{\hookrightarrow}

\newcommand{\wtilde}{\widetilde}
\newcommand{\what}{\widehat}

\newcommand\CounterStep{\addtocounter{thm}{1}\setcounter{equation}{0}}

\newcommand{\factor}[2]{\left. \raise 2pt\hbox{$#1$} \right/\hskip -2pt\raise -2pt\hbox{$#2$}}
\newcommand{\Preprint}[1]{#1}
\newcommand{\Publication}[1]{}

\newcommand{\subversionInfo}{}
\newcommand{\svnid}[1]{}
\newcommand{\approvals}[2][Approval]{}
\renewcommand{\phi}{\varphi}
\allowdisplaybreaks

\author{Daniel Greb}%
\address{Daniel Greb, Essener Seminar für Algebraische Geometrie und Arithmetik,
  Fakultät für Ma\-the\-matik, Universität Duisburg--Essen, 45117 Essen,
  Germany}%
\email{\href{mailto:daniel.greb@uni-due.de}{daniel.greb@uni-due.de}}
\urladdr{\href{http://www.esaga.uni-due.de/daniel.greb}{http://www.esaga.uni-due.de/daniel.greb}}

\author{Stefan Kebekus}%
\address{Stefan Kebekus, Mathematisches Institut, Albert-Ludwigs-Universität
  Freiburg, Ernst-Zermelo-Straße 1, 79104 Freiburg im Breisgau, Germany \&
  Freiburg Institute for Advanced Studies (FRIAS), Freiburg im Breisgau,
  Germany}%
\email{\href{mailto:stefan.kebekus@math.uni-freiburg.de}{stefan.kebekus@math.uni-freiburg.de}}
\urladdr{\href{https://cplx.vm.uni-freiburg.de}{https://cplx.vm.uni-freiburg.de}}

\author{Thomas Peternell}%
\address{Thomas Peternell, Mathematisches Institut, Universität Bayreuth,
  95440~Bayreuth, Germany}%
\email{\href{mailto:thomas.peternell@uni-bayreuth.de}{thomas.peternell@uni-bayreuth.de}}
\urladdr{\href{http://www.komplexe-analysis.uni-bayreuth.de}{http://www.komplexe-analysis.uni-bayreuth.de}}

\author{Behrouz Taji}%
\address{Behrouz Taji, University of Notre Dame, Department of Mathematics, 278
  Hurley, Notre Dame, IN 46556 USA.}%
\email{\href{mailto:btaji@nd.edu}{btaji@nd.edu}}
\urladdr{\href{http://sites.nd.edu/b-taji}{http://sites.nd.edu/b-taji}}

\thanks{DG was partially supported by the DFG-Collaborative Research Center
  SFB/TR 45 ``Periods, moduli spaces and arithmetic of algebraic varieties''.
  SK gratefully acknowledges support through a joint fellowship of the Freiburg
  Institute of Advanced Studies (FRIAS) and the University of Strasbourg
  Institute for Advanced Study (USIAS).  BT was partially supported by the
  DFG-Research Training Group GK1821.  Research was partially completed while SK
  and TP were visiting the National University of Singapore in 2017.}



\title{Harmonic metrics on Higgs sheaves and uniformization of varieties of general type}
\date{\today}

\makeatletter
\hypersetup{
  pdfauthor={\authors},
  pdftitle={\@title},
  pdfsubject={\@subjclass},
  pdfkeywords={\@keywords},
  pdfstartview={Fit},
  pdfpagelayout={TwoColumnRight},
  pdfpagemode={UseOutlines},
  bookmarks,
  colorlinks,
  linkcolor=linkblue,
  citecolor=linkred,
  urlcolor=linkred}
\makeatother

\DeclareMathOperator{\can}{can}
\DeclareMathOperator{\discrep}{discrep}
\DeclareMathOperator{\Div}{Div}

\DeclareMathOperator{\End}{End}

\DeclareMathOperator{\Gal}{Gal}

\DeclareMathOperator{\Higgs}{\sf Higgs}

\DeclareMathOperator{\LSys}{\sf LSys}

\DeclareMathOperator{\pHiggs}{\sf pHiggs}

\DeclareMathOperator{\res}{res}

\DeclareMathOperator{\sLSys}{\sf sLSys}

\DeclareMathOperator{\TPIH}{\sf TPI-Higgs}
\DeclareMathOperator{\TPIL}{\sf TPI-locFree}
\DeclareMathOperator{\TPIR}{\sf TPI-refl}

\theoremstyle{plain}
\newtheorem{factdef}[thm]{Fact and Definition}
\theoremstyle{remark}

\newtheorem{cons}[thm]{Consequence}

\newtheorem{remnot}[thm]{Remark and Notation}

\begin{document}

\begin{abstract}
We prove a criterion for the existence of harmonic metrics on Higgs bundles that
are defined on smooth loci of klt varieties.  As one application, we resolve the
quasi-étale uniformisation problem for minimal varieties of general type to
obtain a complete numerical characterisation of singular quotients of the unit
ball by discrete, co-compact groups of automorphisms that act freely in
codimension one.  As a further application, we establish a nonabelian Hodge
correspondence on smooth loci of klt varieties.
\end{abstract}
\approvals{Behrouz & yes \\ Daniel & yes \\ Stefan & yes \\ Thomas & yes}

\maketitle
\tableofcontents

%
%
\svnid{$Id: 01-intro.tex 1268 2019-02-26 19:31:12Z greb $}

\section{Introduction}
\subversionInfo

\subsection{Main result of this paper}
\approvals{Behrouz & yes \\ Daniel & yes \\ Stefan & yes \\ Thomas & yes}

The core notion of nonabelian Hodge theory, as pioneered by Corlette, Donaldson, Hitchin, and
Simpson, is certainly that of a harmonic bundle.  Most (if not all) important
results of this theory depend on existence results for harmonic metrics in
bundles over projective manifolds, which are usually established using highly
non-trivial analytic methods.

In view of the minimal model program, it is clear that these results should be
studied in the more general context of varieties with terminal or canonical
singularities or even for klt (= Kawamata log terminal) varieties.  In this
context, extending Simpson's theory \cite{MR1179076} from smooth projective
manifolds to minimal models, the paper \cite{GKPT17} established a natural
equivalence between the category of local systems and the category of
semistable, locally free Higgs sheaves with vanishing Chern classes on
projective varieties with klt singularities

\begin{thm}[\protect{\cite[Thm.~1.1]{GKPT17}}]\label{thm:klt}
  Let $X$ be a projective, klt variety.  Then, there exists an equivalence
  between the category of local systems on $X$ and the category of semistable,
  locally free Higgs sheaves with vanishing Chern classes on $X$.  \qed
\end{thm}

However, for geometric applications we also need to understand (flat) locally
free Higgs sheaves on the smooth locus $X_{\reg} $ of a klt variety $X$ which
extend to $X$ as \emph{reflexive} Higgs sheaf rather than as a \emph{locally
  free} Higgs sheaf.  Thanks to the work of Simpson, Jost-Zuo, Mochizuki and
others, there is by now a developed theory of harmonic bundles for non-compact,
quasi-projective manifolds $X°$, establishing the existence of an (essentially
unique) harmonic metric on a given semisimple flat bundle $(E,∇_E)$ on $X°$.  In
particular, $E$ inherits a holomorphic structure and a Higgs field.  In this
context, considering a compactification $X$ of $X°$ by a simple normal crossing
divisor, the notion of a tame and purely imaginary bundle (with respect to the
compactification) play a decisive role.  In case of a projective klt variety $X$
with smooth locus $X° = X_{\reg}$, the situation simplifies to some extent since
$X ∖ X°$ has codimension at least $2$, once one is willing to pay the price that
$X$ is singular.  Our point is that the singularities arising from the minimal
model program are mild enough to still obtain a nonabelian Hodge theory on
$X_{\reg}$ that can be formulated in down-to-earth terms.  Our main result can
be seen as an existence result for harmonic metrics on bundles that are defined
on smooth loci of klt varieties.

\begin{thm}[\protect{= Theorem~\vref{thm:Qflat1new}}]\label{mainthm}
  Let $X$ be a projective klt space of dimension $n ≥ 2$ and $H ∈ \Div(X)$ be
  ample.  Let $(ℰ_{X_{\reg}}, θ_{ℰ_{X_{\reg}}})$ be an reflexive Higgs sheaf on
  $X_{\reg}$ and denote the reflexive extension of $ℰ_{X_{\reg}}$ to $X$ by
  $ℰ_X$.  Then, the following statements are equivalent.
  \begin{enumerate}
  \item The Higgs sheaf $(ℰ_{X_{\reg}}, θ_{ℰ_{X_{\reg}}})$ is (poly)stable with
    respect to $H$ and the $ℚ$-Chern characters satisfy
    $$
    \what{ch}_1(ℰ_X)·[H]^{n-1} = 0 \quad\text{and}\quad
    \what{ch}_2(ℰ_X)·[H]^{n -2} = 0.
    $$
    
  \item The sheaf $ℰ_{X_{\reg}}$ is locally free,
    $(ℰ_{X_{\reg}}, θ_{ℰ_{X_{\reg}}})$ is induced by a tame, purely imaginary
    harmonic bundle whose associated flat bundle is
    (semi)simple.
  \end{enumerate}
\end{thm}

\begin{rem}
  The symbols $\widehat{ch}_•$ that appear in Theorem~\ref{mainthm} denote
  \emph{$ℚ$-Chern characters}, which exist because klt varieties have quotient
  singularities in codimension two.  We refer to
  Section~\ref{subsect:klt_and_Chern} and \cite[Sect.~1.7]{GKT16} for a
  discussion, and to \cite[Sect.~3]{GKPT15} for a proper definition.
\end{rem}
  
We illustrate the usefulness of Theorem~\ref{mainthm} with two applications.

\subsection{Application: a nonabelian Hodge correspondence for local systems on $X_{\reg}$}
\approvals{Behrouz & yes \\ Daniel & yes \\ Stefan & yes \\ Thomas & yes}

The first application pertains to the nonabelian Hodge correspondence.  As
already mentioned, Theorem~\ref{thm:klt} extends the classic correspondence to
projective varieties with klt singularities, in particular to minimal models of
varieties of general type.

As one immediate consequence of Theorem~\ref{mainthm}, we find that our singular
varieties exhibit \emph{two} nonabelian Hodge correspondences, one linking local
systems on the singular space $X$ with locally free Higgs sheaves there, and one
linking local systems on the smooth locus $X_{\reg}$ with Higgs sheaves that may
acquire certain singularities along the singular locus of $X$.  One of the key
features of the theory is that the two correspondences coincide after passing to
a finite quasi-étale cover $Y → X$.  We refer to reader to Section~\ref{sec:klt}
for precise formulations.

\subsection{Application: quasi-étale uniformisation}
\approvals{Behrouz & yes \\ Daniel & yes \\ Stefan & yes \\ Thomas & yes}

One of the cornerstones in the geometry of manifolds of general type is
certainly the Miyaoka-Yau inequality.  In its classical form, it states that
$$
\bigl( 2(n+1)· c_2(X) - n · c_1(X)² \bigr) · [K_X]^{n-2} ≥ 0,
$$
provided the canonical bundle $K_X$ of the $n$-dimensional projective manifold
$X$ is ample.  In case of equality, it is a consequence of the existence of a
Kähler-Einstein metric on $X$, proven by Aubin and Yau, that the universal cover
of $X$ is the unit ball in $ℂ^n$.

Again, in view of the minimal model program it is clear that these results
should be studied in the more general context of varieties with terminal,
canonical, or even klt singularities.  In this context, the Miyaoka-Yau
inequality has been generalised as follows; the formulation again uses $ℚ$-Chern
classes as constructed in \cite[Sect.~3]{GKPT15}.

\begin{thm}[\protect{$ℚ$-Miyaoka-Yau inequality, \cite[Thm.~1.1]{GKPT15} or \cite[Thm.~1.5]{GKT16}}]\label{thm:QMY}
  Let $X$ be an $n$-dimensional, projective, klt variety whose canonical class
  is big and nef.  Then,
  $$
  \bigl( 2(n+1)· \what{c}_2(𝒯_X) - n · \what{c}_1(𝒯_X)² \bigr) · [K_X]^{n-2} ≥ 0.  \eqno\qed
  $$
\end{thm}

The reader is encouraged to also have a look at \cite{GT16}, where related
inequalities for pairs are discussed.

Our new quasi-étale uniformisation result for varieties of general type may then
be formulated as follows.

\begin{thm}[Quasi-étale uniformisation by the unit ball]\label{thm:ballrevisited}
  Let $X$ be an $n$-dimensional, projective, klt variety whose canonical class
  is big and nef.  If equality holds in the $ℚ$-Miyaoka-Yau inequality, i.e.,
  \begin{equation}\label{eq:MY}
    \bigl( 2(n+1)· \what{c}_2(𝒯_X) - n · \what{c}_1(𝒯_X)² \bigr) · [K_X]^{n-2} =0,
  \end{equation}
  then the canonical model of $X$ admits a quasi-étale\footnote{See
    \cite[Def.~2.11]{GKPT15} for the definition of "quasi-étale".}, Galois cover
  $Y → X_{\can}$ by a projective manifold $Y$ whose universal cover is the unit
  ball.
\end{thm}

\begin{cor}[Chern class equality forces quotient singularities]
  In the setting of Theorem~\ref{thm:QMY}, equality in the $ℚ$-Miyaoka-Yau
  Inequality implies that $X_{\can}$ has at worst quotient singularities.  \qed
\end{cor}

\begin{rem}
  Theorem~\ref{thm:ballrevisited} was shown by the authors in \cite{GKPT15}
  under the additional assumption that the variety $X$ be non-singular in
  codimension two; this technical condition allows us to use a significantly
  simpler argument.
\end{rem}

The canonical models that appear in Theorem~\ref{thm:ballrevisited} are
themselves quotients of the unit ball.  The following characterisation is a
minor generalisation of \cite[Thm.~1.3]{GKPT15}; the proof given in
\cite[Sect.~9.1]{GKPT15} applies nearly verbatim and is therefore omitted.

\begin{thm}[Characterisation of singular ball quotients]\label{thm:csbq}
  Let $X$ be a normal, irreducible, compact, complex space of dimension $n$.
  Then, the following statements are equivalent.
  \begin{enumerate}
  \item\label{il:z1} The space $X$ is of the form $𝔹^n/\what{Γ}$ for a discrete,
    co-compact subgroup $\what{Γ} < \Aut_{𝒪}(𝔹^n)$ whose action on $𝔹^n$ is
    fixed-point free in codimension one.

  \item\label{il:z2} The space $X$ is of the form $Y/G$, where $Y$ is a smooth
    ball quotient, and $G$ is a finite group of automorphisms of $Y$ whose
    action is fixed-point free in codimension one.

  \item\label{il:z3} The space $X$ is projective and klt, the canonical divisor
    $K_X$ is ample, and we have equality in \eqref{eq:MY}.  \qed
  \end{enumerate}
\end{thm}

The reader is referred to \cite[Sect.~10]{GKPT15} for a discussion of the
expectations regarding quotients of the ball by properly discontinous group
actions having fixed points in codimension one.

\subsection{Structure of the paper}
\approvals{Behrouz & yes \\ Daniel & yes \\ Stefan & yes \\ Thomas & yes}

Section~\ref{sect:Notations} gathers notation, known results and global
conventions that will be used throughout the paper.

\subsubsection*{Part~\ref*{part:0}}
\approvals{Behrouz & yes \\ Daniel & yes \\ Stefan & yes \\ Thomas & yes}

Part~\ref{part:0} of this paper begins in Section~\ref{sect:harmonic} with a
review of Mochizuki's theory of tame and purely imaginary harmonic bundles on
quasi-projective varieties and discusses them in the particular setting where
the quasi-projective variety is the smooth locus of a klt variety.
Section~\ref{sect:QHiggs} briefly reviews the somewhat delicate notion of a
Higgs sheaf on a singular space, and discusses stability of Higgs bundles that
are defined on the the smooth locus of a klt variety only.  The core of
Part~\ref{part:0} is, however, Section~\ref{sec:flatness} where the the central
existence result for harmonic structures is shown.

\subsubsection*{Part~\ref*{part:II}}
\approvals{Behrouz & yes \\ Daniel & yes \\ Stefan & yes \\ Thomas & yes}

Part~\ref{part:II} of this paper concerns applications.  Section~\ref{sec:klt}
shows in brief how existence of harmonic structures leads to nonabelian Hodge
correspondences pertaining to local systems on the smooth locus of a klt space.
Section~\ref{sec:potbrv} proves our main result on quasi-étale uniformisation,
Theorem \ref{thm:ballrevisited}.  Section~\ref{sec:positivity} studies singular
ball quotients, asking what positivity one might expect in $Ω^{[1]}_X$, and what
hyperbolicity properties might hold in the underlying spaces.

\subsection{Acknowledgements}
\approvals{Behrouz & yes \\ Daniel & yes \\ Stefan & yes \\ Thomas & yes}

We would like to thank numerous colleagues for discussions, including Daniel
Barlet, Oliver Bräunling, Philippe Eyssidieux, Jochen Heinloth, Andreas Höring,
Annette Huber, Shane Kelly, Jong-Hae Keum, Adrian Langer and Jörg Schürmann.  We
also thank the anonymous referee for helpful suggestions for improvement.

%
%
\svnid{$Id: 02-notation.tex 1270 2019-02-27 18:29:41Z peternell $}

\section{Notation and elementary facts}
\label{sect:Notations}
\subversionInfo

\subsection{Global conventions}\label{subsect:global_conventions}
\approvals{Behrouz & yes \\ Daniel & yes \\ Stefan & yes \\ Thomas & yes}

Throughout the present paper, all varieties and schemes will be defined over the
complex numbers.  We follow the notation used in the standard reference books
\cite{Ha77, KM98}, with the exception that klt pairs are assumed to have an
effective boundary divisor, see Section~\ref{subsect:klt_and_Chern} below.

A morphism of vector bundles is always assumed to have constant rank.  Notation
introduced in our previous papers, \cite{GKPT15, GKPT17}, will briefly be
recalled before it is used.

Throughout the paper, we will freely switch between the algebraic and analytic
context if no confusion is likely to arise; sheaves on quasi-projective
varieties will always be algebraic.

\subsection{Nef sheaves}
\approvals{Behrouz & yes \\ Daniel & yes \\ Stefan & yes \\ Thomas & yes}

While positivity notions for vector bundles are well-established in the
literature, we will need these notions also for coherent sheaves.

\begin{defn}[Nef and ample sheaves, \cite{MR670921}]\label{def:possheaf}
  Let $X$ be a normal, projective variety and let $\sS \not = 0$ be a
  non-trivial coherent sheaf on $X$, not necessarily locally free.  We call
  $\sS$ \emph{ample} (resp.\ \emph{nef}) if the locally free sheaf
  $𝒪_{ℙ(\sS)}(1) ∈ \Pic(ℙ(\sS))$ is ample (resp.\ nef).
\end{defn}

We refer the reader to \cite{GrCa60} for the definition of $ℙ(\sS)$, and to
\cite[Sect.~2 and Thm.~2.9]{MR670921} for a more detailed discussion of
amplitude and for further references.  We mention a few elementary facts without
proof.

\begin{fact}[Nef and ample sheaves]\label{fact:1}
  Let $X$ be a normal, projective variety.
  \begin{enumerate}
  \item\label{il:f1-1} Ample sheaves on $X$ are nef.
  \item\label{il:f1-2} A direct sum of sheaves on $X$ if nef iff every summand
    is nef.
  \item\label{il:f1-3} Pull-backs and quotients of nef sheaves are nef.
  \item\label{il:f1-4} A sheaf $ℰ$ is nef on $X$ if and only if for every smooth
    curve $C$ and every morphism $γ: C → X$, the pull-back $γ^* ℰ$ is nef.  \qed
  \end{enumerate}
\end{fact}

\subsection{Connections on complex vector bundles}
\approvals{Behrouz & yes \\ Daniel & yes \\ Stefan & yes \\ Thomas & yes}

Connections on complex vector bundles play a prominent role in this paper.  We
recall two elementary facts that will become relevant later.

\begin{fact}[Extension of bundles with connection]\label{fact:flatext}
  Let $M$ be a $\cC^∞$-manifold and $M° ⊆ M$ an open subset.  Write $\what{π}_1$
  for the profinite completion of the fundamental group, and let
  $\what{ρ} : \what{π}_1(M°) → \what{π}_1(M)$ be the natural morphism induced by
  the inclusion.  If $\what{ρ}$ is isomorphic, then any flat, complex bundle
  $(E°, ∇_{E°})$ on $M°$ admits an extension to a flat, complex bundle
  $(E, ∇_E)$ on $M$.  The extension is unique up to canonical isomorphism.  \qed
\end{fact}

If $γ: X → Y$ is a ramified Galois cover of complex manifolds, if $E$ is a
smooth, complex bundle over $Y$ and $h$ a smooth, Galois-invariant Hermitian
metric on $γ^*E$, it is generally not true that $h$ comes from a smooth metric
on $Y$.  In contrast, the following result asserts that flat connections in
$γ^*E$ do indeed descend once they are invariant.  This is probably known to
experts.  \Preprint{We include a full proof for lack of an adequate
  reference.}\Publication{The preprint version of this paper includes a full
  proof.}

\begin{prop}[Descent of $G$-invariant, flat connections]\label{prop:conndescent}
  Let $γ: X → Y$ be a Galois cover of complex manifolds, with Galois group $G$.
  Let $E_Y$ be a smooth, complex bundle over $Y$ and let $∇_X$ be a flat,
  $G$-invariant connection on $E_X := γ^*E_Y$.  Then, there exists a flat
  connection $∇_Y$ on $E_Y$ such that $∇_X = γ^* ∇_Y$.  \Publication{\qed}
\end{prop}
\Preprint{%
  \begin{proof}
    Decompose $∇_X$ according to type, $∇_X := ∇^{0,1}_X + ∇^{1,0}_X$.  Flatness
    of $∇_X$ allows us to define a $G$-invariant holomorphic structure on $E_X$
    using the differential operator $\bar{∂}_{E_X} := ∇^{0,1}_X$.  Denote the
    locally free sheaf of holomorphic sections by $ℰ_X$.  The operator
    $∇^{1,0}_X$ will then define a $G$-invariant, holomorphic connection
    \begin{equation}\label{eq:hcon1}
      D_X : ℰ_X → ℰ_X ⊗ Ω¹_X.
    \end{equation}
    
    We claim that the holomorphic structure $\bar{∂}_{E_X}$ is the pull-back of
    a holomorphic structure $\bar{∂}_{E_Y}$ that exists on $E_Y$.  This follows
    from \cite[Prop.~4.2]{MR1044586} or \cite[Thm.~4.2.15]{MR2665168}, using the
    observation that if $x ∈ X$ is any point, then the action of the stabiliser
    subgroup $G_x ⊂ G$ on the fibres $(E_X)_x$ is trivial by construction.  As
    before, write $ℰ_Y$ for the locally free sheaf of holomorphic sections.
    
    To conclude, we need to show that that $D_X$ is the pull-back of a
    holomorphic connection $D_Y$ on $ℰ_Y$; the desired connection on $E_Y$ can
    then be defined as $∇_Y := D_Y + \bar{∂}_{E_Y}$.  Since $D_X$ is
    $G$-invariant, and hence a morphism of $G$-sheaves, we may apply the functor
    $γ_*(\,·\,)^G$ to \eqref{eq:hcon1} to obtain a sheaf morphism on $Y$ as
    follows,
    \begin{equation}\label{eq:hcon2}
      γ_*(D_X)^G: γ_* (ℰ_X)^G → γ_*( ℰ_X ⊗ Ω¹_X)^G.
    \end{equation}
    Both sides of \eqref{eq:hcon2} have elementary descriptions in terms of
    known objects on $Y$,
    \begin{align*}
      γ_* (ℰ_X)^G & ≅ ℰ_Y && \text{Construction} \\
      γ_*(ℰ_X ⊗ Ω¹_X)^G & ≅ ℰ_Y ⊗ γ_* (Ω¹_X)^G && \text{Equivar.\ projection formula, \cite[Lem.~4.9]{GT13}} \\
                  & ≅ ℰ_Y ⊗ Ω¹_Y && \text{\cite[Thm.~1]{MR1451789}}
    \end{align*}
    With these isomorphisms in place, the map \eqref{eq:hcon2} is identified
    with a sheaf morphism
    $$
    D_Y: ℰ_Y → ℰ_Y ⊗ Ω¹_Y.
    $$
    Elementary computations show that $D_Y$ is indeed a holomorphic connection,
    and that its pull-back to $X$ equals $D_X$.
  \end{proof}
}

\subsection{Higgs sheaves}
\approvals{Behrouz & yes \\ Daniel & yes \\ Stefan & yes \\ Thomas & yes}

Let $X$ be a normal variety or normal complex space.  Following the notation
introduced in \cite[Def.~5.1]{GKPT15}, a \emph{Higgs sheaf} is a pair $(ℰ, θ)$
of a coherent sheaf $ℰ$ of $𝒪_X$-modules, together with an $𝒪_X$-linear sheaf
morphism $θ : ℰ → ℰ ⊗ Ω^{[1]}_X$, called \emph{Higgs field}, such that the
induced morphism $ℰ → ℰ ⊗ Ω^{[2]}_X$ vanishes.  We refer the reader to
\cite[Sect.~5]{GKPT15} for related notions, including the definition of a Higgs
$G$-sheaf, and slope stability of Higgs sheaves with respect to nef divisor
classes.

\subsubsection{Categories used in the nonabelian Hodge correspondence}
\approvals{Behrouz & yes \\ Daniel & yes \\ Stefan & yes \\ Thomas & yes}

If $X$ is a projective normal variety, we often consider locally free Higgs
sheaves $(ℰ,θ)$ on $X$ having the property that there exists an ample divisor
$H ∈ \Div(X)$ such that
\begin{itemize}
\item the Higgs sheaf $(ℰ, θ)$ is semistable with respect to $H$, and
\item the Chern characters of $ℰ$ satisfy
  $ch_1(ℰ)·[H]^{n-1} = ch_2(ℰ)·[H]^{n-2} = 0$.
\end{itemize}
These sheaves form a category, which plays a central role in the nonabelian
Hodge correspondence for klt spaces, \cite[Sect.~3.1]{GKPT17}.  We denote this
category by $\Higgs_X$.

\subsection{KLT spaces and $ℚ$-Chern classes}\label{subsect:klt_and_Chern}
\approvals{Behrouz & yes \\ Daniel & yes \\ Stefan & yes \\ Thomas & yes}

A \emph{klt pair} $(X,Δ)$ consists of a normal variety $X$ and a Weil
$ℚ$-divisor $Δ= \sum_i a_i D_i$ with $a_i ∈ ℚ ∩ (0,1)$ such that $K_X + Δ$ is
$ℚ$-Cartier and such that $\discrep(X,Δ) > -1$, where the discrepancy
$\discrep(X,Δ)$ is defined in \cite[Def.~2.28]{KM98} using
\cite[Def.~2.25]{KM98}.  In contrast to \cite[Def.~2.34]{KM98} we do not allow
non-effective boundary divisors $Δ$.

\begin{defn}
  A normal, quasi-projective variety $X$ is called \emph{klt space} if there
  exists an effective $ℚ$-divisor $Δ$ that makes the pair $(X,Δ)$ klt.
\end{defn}

\begin{rem}
  An klt space $X$ is normal, hence non-singular in codimension one.  Better
  still, recall from \cite[Sect.~3.7]{GKPT15} that klt spaces have quotient
  singularities in codimension two.  If $ℰ_X$ is any coherent reflexive sheaf on
  $X$, this allows us to define $ℚ$-Chern classes $\what{c}_1(ℰ_X)$ and
  $\what{c}_1(ℰ_X)$, as well as $ℚ$-Chern characters
  $$
  \what{ch}_1(ℰ_X) = \what{c}_1(ℰ_X) \quad \text{and} \quad \what{ch}_2(ℰ_X) :=
  \frac{1}{2}\bigl(\what{c}_1(ℰ_X)² - 2· \what{c}_2(ℰ_X)\bigr).
  $$
\end{rem}

\part{Existence of harmonic bundle structures}
\label{part:0}

%
%
\svnid{$Id: 03-harmonic.tex 1269 2019-02-27 07:58:09Z taji $}

\section{Harmonic bundles}\label{sect:harmonic}
\approvals{Behrouz & yes \\ Daniel & yes \\ Stefan & yes \\ Thomas & yes}

Harmonic bundles are key tools in nonabelian Hodge theory that provide the link
between flat structures and Higgs fields.  We briefly recall the definition,
explain relevant properties, and recall the notion of ``tameness'' that is used
to establish a good theory in non-compact, compactifiable situations.  Finally,
we prove a boundedness result for Higgs bundles admitting a tame and purely
imaginary harmonic structure.

\begin{factdef}[\protect{Harmonic bundle, cf.~\cite[Sect.~1]{MR3087348}}]\label{def:harm}
  Let $M$ be a complex manifold.  Consider a tuple $𝔼 = (E, \bar{∂}, θ, h)$
  comprised of the following data.
  \begin{itemize}
  \item A holomorphic vector bundle $(E, \bar{∂})$ and a Hermitian metric $h$ on
    $E$.
  \item A Higgs field $θ : ℰ → ℰ ⊗ Ω¹_X$, where $ℰ = \ker \bar{∂}$ is the sheaf
    of holomorphic sections.
  \end{itemize}
  By minor abuse of notation, we will also write $θ$ for the induced $𝒜⁰$-linear
  morphism $θ : 𝒜⁰(E) → 𝒜^{1,0}(E)$.  Let $θ^h : 𝒜⁰(E) → 𝒜^{0,1}(E)$ be the
  adjoint\footnote{In local coordinates, if $θ = \sum_k θ_k\, d z_k$, then
    $θ^h = \sum_k θ_k^{*}\, d \bar{z}_k$, where $θ_k^*$ is the adjoint of $θ_k$
    with respect to the metric $h$.} of $θ$ with respect to the metric $h$, and
  let $∂$ be the $(1,0)$-part of the unique Chern-connection compatible with
  both the metric $h$ and the complex structure $\bar{∂}$.  Then,
  \begin{equation}\label{eq:Harm}
    ∇_𝔼 := ∂ + \bar{∂} + θ + θ^h,
  \end{equation}
  is a connection.  The tuple $𝔼$ is called a \emph{harmonic bundle} if $∇_𝔼$ is
  flat.
\end{factdef}

\begin{notation}[Flat bundles and local systems associated with harmonic bundles]
  Given a harmonic bundle $𝔼 = (E, \bar{∂}, θ, h)$ as in Fact and
  Definition~\ref{def:harm}, we denote the associated flat bundle by $(E, ∇_𝔼)$
  and write $\aE ∈ \LSys_M$ for the local system.
\end{notation}

\begin{notation}[Sheaves admitting a harmonic bundle structure]\label{def:carries}
  Let $M$ be a complex manifold.
  \begin{itemize}
  \item Given a locally free sheaf $ℰ$ on $M$ with associated holomorphic bundle
    $(E, \bar{∂})$, we say that $ℰ$ \emph{admits a harmonic bundle structure} if
    there exists a harmonic bundle of the form $(E, \bar{∂}, θ, h)$.
    
  \item Given a locally free Higgs sheaf $(ℰ, θ)$ on $M$ with associated
    holomorphic bundle $(E, \bar{∂})$, we say that $(ℰ, θ)$ \emph{admits a
      harmonic bundle structure} if there exists a harmonic bundle of the form
    $(E, \bar{∂}, θ, h)$.
    
  \item We say that a flat bundle $(E,∇)$ \emph{admits a harmonic bundle
      structure} if there exists a harmonic bundle $𝔼 = (E, \bar{∂}, θ, h)$ such
    that $∇ = ∇_𝔼$.
  \end{itemize}
  If $X$ is a smooth, quasi-projective manifold and $ℰ$ a locally free sheaf on
  $X$, we say that $ℰ$ admits a harmonic bundle structure if its analytification
  $ℰ^{an}$ on the complex manifold $X^{an}$ admits a harmonic bundle structure.
  Analogously for algebraic Higgs bundles $(ℰ, θ)$ on $X$.
\end{notation}

\begin{rem}\label{rem:restriction}
  In the setup of Notation~\ref{def:carries}, let $N$ be a complex submanifold
  of $M$.  Assume that $(ℰ, θ)$ admits a harmonic bundle structure.  Then, an
  easy local computation shows that the locally free Higgs sheaf $(ℰ, θ)|_N$
  admits a harmonic bundle structure given by restriction.
\end{rem}

\subsection{Tame and purely imaginary bundles}
\approvals{Behrouz & yes \\ Daniel & yes \\ Stefan & yes \\ Thomas & yes}

In order to study Higgs bundles on quasi-projective, non-projective varieties,
we consider ``tame'' harmonic bundles.  These are harmonic bundles on the
complement of a divisor whose growth near the divisor is sufficiently
controlled.

\subsubsection{Basic definitions}
\approvals{Behrouz & yes \\ Daniel & yes \\ Stefan & yes \\ Thomas & yes}

The following is not Simpson's original definition of ``tameness'', but is
equivalent to it.

\begin{defn}[\protect{Tame harmonic bundle, \cite[Sect.~22.1 and Lem.~22.1]{MR2283665}}]\label{def:tame}
  Let $M$ be a complex manifold, let $D ⊂ M$ be a divisor with simple normal
  crossings, and let $𝔼 = (E, \bar{∂}, θ, h)$ be a harmonic bundle on $M ∖ D$.
  The harmonic bundle $𝔼$ is called \emph{tame with respect to $(M,D)$} if there
  exists a locally free sheaf $ℰ_M$ on $M$, a sheaf morphism
  $$
  θ_M: ℰ_M → ℰ_M ⊗ Ω¹_M(\log D)
  $$
  and an isomorphism $ℰ_M|_{M ∖ D} ≅ ℰ$ that identifies $θ_M|_{M ∖ D}$ with $θ$.
  We call $(ℰ_M, θ_M)$ an \emph{extension of $(ℰ, θ)$}.
\end{defn}

\begin{factdef}[\protect{Purely imaginary bundles, \cite[Lem.~22.2]{MR2283665}}]
  In the setup of Definition~\ref{def:tame}, assume that $(E, \bar{∂}, θ, h)$ is
  tame, and let $(ℰ_M, θ_M)$ be an extension of $(ℰ, θ)$.  If $D_i ⊂ D$ is any
  component and $x ∈ D_i$ is any point, consider the residue and its restriction
  to $x$,
  $$
  \res_{D_i} θ_M ∈ \End \big( ℰ_M|_{D_i}\bigr) \quad\text{and}\quad (\res_{D_i}
  θ_M)|_x ∈ \End \big( ℰ_M|_x \bigr).
  $$
  Then, then the set of eigenvalues of $(\res_{D_i} θ_M )|_x$ is independent of
  the choice of $(ℰ_M,θ_M)$.

  The harmonic bundle $(E, \bar{∂}_E, θ, h)$ is called \emph{purely imaginary
    with respect to $(M,D)$} if all eigenvalues of the residues of $θ_M$ along
  the irreducible components of $D$ are purely imaginary for one (equivalently
  any) extension $(ℰ_M, θ_M)$ of $(ℰ, θ)$.  \qed
\end{factdef}

It is important to notice that the notion of tame purely imaginary bundles does
not depend on the compactification.

\begin{factdef}[\protect{Tame and purely imaginary bundles on quasi-projective varieties, \cite[Lem.~25.29]{MR2283665} and \cite[Cor.~8.7]{MR2281877}}]\label{factdef:harmonic_independent}
  Let $X$ be a smooth, quasi-projective variety and let
  $𝔼 := (E, \bar{∂}, θ, h)$ be a harmonic bundle on $X^{an}$.  Let
  $\overline{X}_1$ and $\overline{X}_2$ be two smooth, projective
  compactifications of $X$ such that $D_• := \overline{X}_• ∖ X$ are snc
  divisors.  Then, $𝔼$ is tame and purely imaginary with respect to
  $(\overline{X}_1,D_1)$ if and only if $𝔼$ is tame and purely imaginary with
  respect to $(\overline{X}_2, D_2)$.  We can therefore speak about \emph{tame
    and purely imaginary harmonic bundles} on $X^{an}$.  \qed
\end{factdef}

\begin{rem}[Automatic algebraicity I]\label{rem:autoalg}
  In the setup of Fact and Definition~\ref{factdef:harmonic_independent},
  setting as usual $ℰ := \ker \bar{∂}$, we may apply Serre's GAGA to the
  extension of $ℰ$ to a smooth projective simple normal crossings
  compactification of $X$ (as in Definition~\ref{def:tame}) to see that the
  holomorphic Higgs sheaf $(ℰ, θ)$ on $X^{an}$ can be endowed with an algebraic
  structure.
\end{rem}

\begin{remnot}[Uniqueness of the algebraic structure]\label{rem:autoalg1a}
  We will often consider the setup of Fact and
  Definition~\ref{factdef:harmonic_independent} in a situation where the smooth
  variety $X$ is a big open subset of a normal projective variety
  $\overline{X}$.  If $ℰ$ is any locally free sheaf on $X^{an}$ that admits a
  tame and purely imaginary harmonic bundle structure.  Then, $ℰ$ can be endowed
  with an algebraic structure and hence has a coherent extension to
  $\overline{X}^{an}$.  By \cite[Thm.~1]{MR0212214}, the analytic sheaf $ℰ$ will
  then have a \emph{unique} reflexive extension to $X^{an}$.  It then follows
  from GAGA that the induced algebraic structure on $ℰ$ is unique up to
  isomorphism; the same holds for the Higgs field as well.
  
  In this situation, we simply say that $ℰ$ and $(ℰ, θ)$ are algebraic, and
  freely switch between the analytic and the algebraic category if no confusion
  seems likely.
\end{remnot}

\begin{rem}\label{rem:tame_semisimple}
  Flat sheaves admitting a tame and purely imaginary harmonic bundle structure
  are necessarily semisimple, see \cite[Prop.~22.15]{MR2283665}.
\end{rem}

\begin{notation}[Bundles admitting tame, purely imaginary harmonic structures]\label{not:TPI}
  If $X$ is a smooth, quasi-projective variety, write
  $$
  \TPIL_X
  $$
  for the family of (algebraic) isomorphism classes of locally free (algebraic)
  sheaves on $X$ that admit a tame, purely imaginary harmonic bundle structure.
  Write
  $$
  \TPIH_X
  $$
  for the family of algebraic isomorphism classes of locally free Higgs sheaves
  on $X$ that admit a tame, purely imaginary harmonic bundle structure.  Abusing
  notation, we write $(ℰ, θ) ∈ \TPIH_X$ to indicate that a given Higgs sheaf
  $(ℰ, θ)$ on $X$ is locally free and admits a tame, purely imaginary harmonic
  bundle structure $𝔼$.  If $(ℰ, θ) ∈ \TPIH_X$ and if $X$ is a big open subset
  of a normal projective variety $\overline{X}$, we also say that $(ℰ, θ)$ is
  \emph{induced} by $𝔼$.
\end{notation}

\begin{lem}[Flat subsheaves in tame and purely imaginary harmonic bundles]\label{lem:flshtpihb}
  Let $X$ be a smooth, quasi-projective variety and let
  $𝔼 = (E,\bar{∂}_E, θ, h)$ be a tame and purely imaginary harmonic bundle on
  $X^{an}$ with induced flat connection $∇_𝔼$.  If $F ⊆ E$ is any complex
  subbundle that is invariant with respect to $∇_𝔼$, then $\bar{∂}$ restricts to
  equip $F$ with the structure of a Higgs-invariant, holomorphic subbundle of
  $(E, \bar{∂})$.
\end{lem}
\begin{proof}
  It suffices to consider the case where $F$ with its induced flat structure is
  irreducible.  But there, the description of the tame and purely imaginary
  harmonic bundle $𝔼$ in \cite[Lem.~A.13]{MR2310103} immediately implies that
  the metric complement $F^\perp$ of $F$ is likewise invariant with respect to
  $∇_𝔼$.  The claim then follows from the description of the operators $\bar{∂}$
  and $θ$ in terms of $∇_𝔼$ and $h$, cf.~\cite[p.~13]{MR1179076}.
  \Publication{The preprint version of this paper spells out all
    details.}\Preprint{In detail, decompose the connection $∇_𝔼: 𝒜⁰(E) → 𝒜¹(E)$
    according to type, $∇_𝔼 = ∇^{1,0}_𝔼 + ∇^{0,1}_𝔼$, and observe that both $F$
    and $F^\perp$ are invariant with respect to each operator $∇^{•,•}_𝔼$.  Let
    $δ^{1,0}$ and $δ^{0,1}$ be the unique operators of type $(1,0)$ and $(0,1)$
    such that $∇^{1,0}_𝔼 + δ^{0,1}$ and $∇^{0,1}_𝔼 + δ^{1,0}$ are metric
    connections for the Hermitian metric $h$.  Spelled out\footnote{By standard
      abuse of notation, the pairings $h(•, •)$ that appear in the right hand
      side of \eqref{eq:xb.nx} refer to the natural extensions of the Hermitian
      metric $h : 𝒜⁰(E) ⨯ 𝒜⁰(E) → 𝒜⁰$ to sesquilinear pairings between sheaves
      of forms with values in $E$, that is, $h : 𝒜^p(E) ⨯ 𝒜^q(E) → 𝒜^{p+q}$.},
    this means that for all $σ$ and $τ$ in $𝒜⁰(E)$, we have an equality in $𝒜¹$,
    \begin{equation}\label{eq:xb.nx}
      \begin{aligned}
        dh(σ,τ) & = h \bigl(∇^{1,0}_𝔼σ + δ^{0,1}σ,\, τ\bigr) + h\bigl(σ,\, ∇^{1,0}_𝔼τ + δ^{0,1}τ \bigr) \\
        & = \underbrace{h \bigl(∇^{1,0}_𝔼σ,\, τ\bigr) + h\bigl(σ,\, δ^{0,1}τ \bigr)}_{\text{type $(1,0)$}} + \underbrace{h\bigl(δ^{0,1}σ,\, τ\bigr) + h\bigl(σ,\, ∇^{1,0}_𝔼τ \bigr);}_{\text{type $(0,1)$}}
      \end{aligned}
    \end{equation}
    analogous equations hold for $∇^{0,1}_𝔼 + δ^{1,0}$.  Next, recall from
    \cite[p.~13]{MR1179076} that
    $$
    \bar{∂} = \frac{1}{2}·\bigl(∇^{0,1}_𝔼 + δ^{0,1}\bigr) \qquad\text{and}\qquad
    θ = \frac{1}{2}·\bigl(∇^{1,0}_𝔼 - δ^{1,0}\bigr).
    $$
    To show that $F$ is invariant with respect to $\bar{∂}$, or equivalently
    that $F$ is a holomorphic subbundle of $(E,\bar{∂})$, it will therefore
    suffice to show that $F$ is invariant with respect to $δ^{0,1}$.  In plain
    words, we need to show that given any section $σ ∈ 𝒜⁰(F)$, then
    $δ^{0,1}(σ) ∈ 𝒜^{0,1}(F)$.  Equivalently, we need to show that given any
    tuple of sections, $σ ∈ 𝒜⁰(F)$ and $τ ∈ 𝒜⁰(F^\perp)$, then
    $h (δ^{0,1}(σ), τ) = 0$ in $𝒜^{0,1}$.  This, however, follows by looking at
    the $(0,1)$-part of Equation~\eqref{eq:xb.nx},
    $$
    0 = (0,1)\text{-part of } dh(σ,τ) = h\bigl(δ^{0,1}σ,\, τ\bigr) +
    \underbrace{h\bigl(σ,\, ∇^{1,0}_𝔼τ \bigr)}_{\mathclap{=0\text{ since
          $F^\perp$ is invariant under } ∇_𝔼}}.
    $$
    Looking at $∇^{0,1}_𝔼 + δ^{1,0}$ instead of $∇^{1,0}_𝔼 + δ^{0,1}$, we see
    that $F$ is also invariant with respect to the operator $δ^{1,0}$ and hence
    with respect to $θ$.  The holomorphic subbundle $F$ of $(E,\bar{∂}_E)$ is
    therefore Higgs-invariant, as claimed.}
\end{proof}

\subsubsection{Existence and uniqueness}
\approvals{Behrouz & yes \\ Daniel & yes \\ Stefan & yes \\ Thomas & yes}

If $X$ is smooth and quasi-projective, then a result of Jost-Zuo \cite{JZ97}
implies that every semisimple flat bundle on $X$ admits a tame and purely
imaginary metric, which is essentially unique\footnote{See the argument in
  \cite[Thm.~25.28]{MR2283665}.}.  We summarise the results relevant for us in
the following theorem, see \cite[Lem.~A.13]{MR2310103} and further references
given there.

\begin{thm}[Existence and uniqueness of harmonic structures]\label{thm:JZM}
  Let $X$ be a smooth quasi-projective variety.  Then, every semisimple flat
  vector bundle $(E,∇_E)$ on $X$ admits a tame, purely imaginary harmonic bundle
  structure $𝔼 = (E, \bar{∂}, θ, h)$.  The metric $h$ is unique up to flat
  automorphisms of $E$, and, as a consequence, the operators in the induced
  decomposition \eqref{eq:Harm} are independent of the choice of such $h$.  \qed
\end{thm}

The following consequences will be used later.

\begin{cor}[Higgs bundles determined by induced connection]\label{cor:uniq}
  Let $X$ be a smooth, quasi-projective variety and let
  $𝔼 := (E, \bar{∂}_E, θ_ℰ, h_E)$ and $𝔽 := (F, \bar{∂}_F, θ_ℱ, h_F)$ be two
  tame, purely imaginary harmonic bundles on $X^{(an)}$, with associated locally
  free sheaves $ℰ$ and $ℱ$.  Assume that the flat bundles $(E, ∇_𝔼)$ and
  $(F, ∇_𝔽)$ are isomorphic.  Then, also the corresponding holomorphic Higgs
  bundles are holomorphically isomorphic, $(ℰ, θ_ℰ) ≅ (ℱ, θ_ℱ)$.
\end{cor}
\begin{proof}
  By assumption, there exists a smooth isomorphism $Φ : E → F$ such that
  $Φ^*∇_{𝔽} = ∇_𝔼$.  The pull-back $Φ^*𝔽 = (E, Φ^* \bar{∂}_F, Φ^*θ_ℱ, Φ^*h_F)$
  will thus equip $E$ with a second tame, purely imaginary harmonic bundle
  structure, whose associated flat connection $∇_{Φ^*𝔽}$ equals $∇_𝔼$.  But then
  it follows from Theorem~\ref{thm:JZM} that the differential operators in the
  two harmonic bundle structures agree: $\bar{∂}_E = Φ^* \bar{∂}_F$ and
  $θ_ℰ = Φ^*θ_ℱ$.  In other words, $Φ$ is holomorphic and induces an holomorphic
  isomorphism of Higgs bundles.
\end{proof}

\begin{rem}[Automatic algebraicity II]\label{rem:autoalgII}
  In the setting of Corollary~\ref{cor:uniq}, recall from
  Remark~\ref{rem:autoalg} that the Higgs sheaves $(ℰ, θ_ℰ)$ and $(ℱ, θ_ℱ)$ are
  in fact algebraic.  If $X$ is isomorphic to a big open subset in a normal
  variety, then the holomorphic isomorphism given in Corollary~\ref{cor:uniq}
  extends to an isomorphism between reflexive closures, and is therefore
  likewise algebraic.
\end{rem}

\begin{cor}[Extension of harmonic bundles from hyperplanes]\label{cor:g1}
  Let $X$ be a normal, projective variety of dimension $\dim X > 2$, and let
  $H ∈ \Div(X)$ be ample.  If $m ≫ 0$ is large enough and $D ∈ |m·H|$ is
  general, then $D$ is normal, $D_{\reg} = D ∩ X_{\reg}$, and the restriction
  map $\TPIH_{X_{\reg}} → \TPIH_{D_{\reg}}$ is surjective.
\end{cor}
\begin{proof}
  If $m ≫ 0$ is large enough, then $|m·H|$ is basepoint free, we have
  $D_{\reg} = D ∩ X_{\reg}$ by Bertini, and the Lefschetz hyperplane theorem for
  fundamental groups, \cite[Thm.\ in Sect.~II.1.2]{GoreskyMacPherson} implies
  that the natural morphism $π_1(D_{\reg}) → π_1(X_{\reg})$ is isomorphic.  Now,
  assuming we are given a $(ℰ_{D_{\reg}}, θ_{ℰ_{D_{\reg}}}) ∈ \TPIH_{D_{\reg}}$,
  we need to show that there exists
  $(ℱ_{X_{\reg}}, θ_{ℱ_{X_{\reg}}}) ∈ \TPIH_{X_{\reg}}$ such that
  \begin{equation}\label{eq:ghjkg}
    (ℰ_{D_{\reg}}, θ_{ℰ_{D_{\reg}}}) ≅ (ℱ_{X_{\reg}},
    θ_{ℱ_{X_{\reg}}})|_{D_{\reg}}.
  \end{equation}
  To this end, choose a tame, purely imaginary harmonic bundle $𝔼_{D_{\reg}}$
  inducing $(ℰ_{D_{\reg}}, θ_{ℰ_{D_{\reg}}})$, which exists by assumption, and
  recall from Remark \ref{rem:tame_semisimple} that the induced local system
  $\aE_{D_{\reg}} ∈ \LSys_{D_{\reg}}$ is semisimple.  Then, using the Lefschetz
  Theorem we can extend $\aE_{D_{\reg}}$ to a semisimple local system
  $\aF_{X_{\reg}} ∈ \LSys_{X_{\reg}}$ in a unique manner.  By the Jost-Zuo
  existence result for harmonic structures, Theorem~\ref{thm:JZM}, there will
  then exist a tame, purely imaginary harmonic bundle $𝔽_{X_{\reg}}$ on
  $X_{\reg}$ that induces $\aF_{X_{\reg}}$.  If
  $(ℱ_{X_{\reg}}, θ_{ℱ_{X_{\reg}}})$ is the associated Higgs bundle on
  $X_{\reg}$, Corollary~\ref{cor:uniq} together with Remark~\ref{rem:autoalgII}
  gives the desired (algebraic) isomorphism~\eqref{eq:ghjkg}.
\end{proof}

\subsection{TPI-Harmonic bundles on klt spaces}
\approvals{Behrouz & yes \\ Daniel & yes \\ Stefan & yes \\ Thomas & yes}%
\label{ssec:poticwXs0}

Let $X$ be a projective klt space.  Using the nonabelian Hodge correspondence
for locally free Higgs sheaves on klt spaces, \cite[Thm.~3.4]{GKPT17}, we will
show boundedness of the family $\TPIH_{X_{\reg}}$.  As a consequence, we will
obtain in Corollary~\ref{cor:restr} a criterion for the existence of harmonic
structures: a given Higgs bundle on $X_{\reg}$ admits a harmonic structure if
and only if its restriction to the smooth locus of a general hypersurface does.
The proof uses the existence, for every klt space $X$, of a ``maximally
quasi-étale cover''.  This is a quasi-étale cover $γ: Y → X$ such that the
natural map of étale fundamental groups,
$\widehat{π}_1(Y_{\reg}) → \widehat{π}_1(Y)$, is isomorphic.  The existence of
such a cover was established in \cite[Thm.~1.5]{GKP13}.

\begin{prop}[Pull-back of Higgs sheaves to maximally quasi-étale cover]\label{prop:p1}
  Let $X$ be a projective klt space.  Let $γ: Y → X$ be a maximally quasi-étale
  cover, write $X° := X_{\reg}$, $Y° := γ^{-1}(X°)$ and consider the étale
  morphism $δ : Y° → X°$ given as the restriction of $γ$.  Given
  $(ℰ_{X°}, θ_{ℰ_{X°}}) ∈ \TPIH_{X°}$, there exists $(ℱ_Y, θ_{ℱ_Y}) ∈ \Higgs_Y$
  and an isomorphism
  \begin{equation}\label{eq:lp1}
    (ℱ_Y, θ_{ℱ_Y})|_{Y°} ≅ δ^* \bigl(ℰ_{X°}, θ_{ℰ_{X°}} \bigr).
  \end{equation}
  In particular, if $ℰ_X$ denotes the reflexive extension of $ℰ_{X°}$ to $X$,
  then $γ^{[*]} ℰ_X ≅ ℱ_Y$ is locally free and all its Chern classes vanish.
\end{prop}

\begin{remnot}\label{rem:p1}
  In the setting of Proposition~\ref{prop:p1}, if $γ$ is Galois with group $G$,
  then $(ℱ_Y, θ_{ℱ_Y})$ carries the structure of a Higgs $G$-sheaf.  Writing
  $(ℱ_{Y°}, θ_{ℱ_{Y°}}) := (ℱ_Y, θ_{ℱ_Y})|_{Y°}$, the invariant push-forward of
  the $G$-equivariant Higgs field $θ_{ℱ_{Y°}}$,
  $$
  (δ_* θ_{ℱ_{Y°}})^G : \underbrace{\bigl( δ_* ℱ_{Y°} \bigr)^G}_{≅ ℰ°} →
  \underbrace{\bigl( δ_* (ℱ_{Y°} ⊗ Ω¹_{Y°}) \bigr)^G}_{≅ ℰ_{X°} ⊗ Ω¹_{X°}
    \text{ since $δ$ is étale}},
  $$
  is immediately identified with the Higgs field $θ_{ℰ_{X°}}$.  Abusing
  notation, we write $(ℰ_{X°}, θ_{ℰ_{X°}}) = δ_* (ℱ_{Y°}, θ_{ℱ_{Y°}})^G$ in this
  context.
\end{remnot}

\begin{proof}[Proof of Proposition~\ref*{prop:p1}]
  Write $(ℱ_{Y°}, θ_{ℱ_{Y°}}) := δ^* (ℰ_{X°}, θ_{ℰ_{X°}})$ and choose a tame,
  purely imaginary harmonic bundle structure $𝔼_{X°}$ for
  $(ℰ_{X°}, θ_{ℰ_{X°}})$.  Recall from \cite[Lem.~25.29]{MR2283665} that the
  pullback $𝔽_{Y°} := δ^* 𝔼_{X°}$ is a tame and purely imaginary harmonic bundle
  structure for $(ℱ_{Y°}, θ_{ℱ_{Y°}})$.  The induced local system $\aF_{Y°}$ is
  semisimple by \cite[Prop.~22.15]{MR2283665}.  The assumption that $Y$ is
  maximally quasi-étale implies that $\aF_{Y°}$ extends from $Y°$ to a
  semisimple local system $\aF_Y ∈ \LSys_Y$ that is defined on all of $Y$, see
  \cite[Thm.~1.2b]{MR0262386} or \cite[Sect.~8.1]{GKP13}.  In particular, the
  nonabelian Hodge correspondence for klt spaces, \cite[Thm.~3.4]{GKPT17},
  applies to yield a locally free Higgs sheaf
  $(ℱ_Y, θ_{ℱ_Y}) := η_Y(\aF_Y) ∈ \Higgs_Y$.  More is true: We have seen in
  \cite[Prop.~3.11]{GKPT17} that $(ℱ_Y, θ_{ℱ_Y})|_{Y°}$ admits a tame, purely
  imaginary harmonic bundle structure $𝔽'_{Y°}$ whose associated local system
  $\aF'_{Y°}$ is isomorphic to $\aF_{Y°}$.  Corollary~\ref{cor:uniq} thus yields
  the desired isomorphism \eqref{eq:lp1}.
\end{proof}

\begin{cor}[Boundedness of $\TPIH$]\label{cor:p2}
  Let $X$ be a projective klt space.  Then, the families $\TPIH_{X_{\reg}}$ and
  $\TPIL_{X_{\reg}}$ are bounded, and so is the family
  $$
  \TPIR_X := \text{iso.\ classes of reflexive sheaves $ℱ$ on $X$ with
    $ℱ|_{X_{\reg}} ∈ \TPIL_{X_{\reg}}$}.
  $$
\end{cor}
\begin{proof}
  Choose a maximally quasi-étale cover $γ: Y → X$, Galois with group $G$, and
  use the same notation $X°$, $Y°$ and $δ$ as in Proposition~\ref{prop:p1}.
  Recall from \cite[Cor.~7.2]{GKPT17} that the family $\Higgs_Y$ is bounded.
  Standard arguments, which we leave to the reader, show that the following
  families are likewise bounded
  \begin{align*}
    \bA_1 & := \bigl\{ \text{loc.~free Higgs $G$-sheaves on $Y$ whose underlying Higgs sheaf is in} \Higgs_Y \bigr\} \\
    \bA_2 & := \bigl\{ \text{loc.~free Higgs sheaves on $X°$ of the form $δ_* ((ℱ_Y, θ_{ℱ_Y})|_{Y°} )^G$, $(ℱ_Y, θ_{ℱ_Y})$ in $\bA_1$} \bigr\}.
  \end{align*}
  But we have seen in Proposition~\ref{prop:p1} and Remark~\ref{rem:p1} that
  every locally free Higgs sheaf in $\TPIH_{X_{\reg}}$ appears in $\bA_2$.  This
  proves boundedness of $\TPIH_{X_{\reg}}$ and $\TPIL_{X_{\reg}}$.  Finally, by
  considering the push-forward from $X_{\reg}$ to $X$ and by using reflexivity
  of its members we conclude that $\TPIR_X$ is also bounded.
\end{proof}

The following corollary is now a direct consequence of
Remark~\ref{rem:restriction}, of the boundedness result above, of
Corollary~\ref{cor:g1}, and of the iterated Bertini-type theorem for bounded
families, \cite[Prop.~7.3]{GKPT17}.  We emphasise that it contains a criterion
for a reflexive Higgs sheaf to be locally free.

\begin{cor}[Restriction criterion for $\TPIH$]\label{cor:restr}
  Let $X$ be a projective klt space of dimension $n > 2$, let $H ∈ \Div(X)$ be
  ample, and let $(ℰ_{X_{\reg}}, θ_{ℰ_{X_{\reg}}})$ be a reflexive Higgs sheaf
  on $X_{\reg}$.  If $m ≫ 0$ is large enough, then there exists a dense, open
  $\aB° ⊆ |m·H|$ such that every hypersurface $D ∈ \aB°$ is a klt space,
  satisfies $D_{\reg} = D ∩ X_{\reg}$, and such that
  $$
  (ℰ_{X_{\reg}}, θ_{ℰ_{X_{\reg}}}) ∈ \TPIH_{X_{\reg}} \;\;\; ⇔ \;\;\;\;
  (ℰ_{X_{\reg}}, θ_{ℰ_{X_{\reg}}})|_{D_{\reg}} ∈ \TPIH_{D_{\reg}}.
  $$
  In plain words, the reflexive Higgs sheaf $(ℰ_{X_{\reg}}, θ_{ℰ_{X_{\reg}}})$
  is locally free and admits a tame and purely imaginary harmonic bundle
  structure if and only if its restriction to $D_{\reg}$ is locally free and
  admits a tame and purely imaginary harmonic bundle structure.  \qed
\end{cor}
\begin{proof}
  The conditions on the singularities of the hypersurface $D$ stated above are
  met by a Zariski-open subset of any basepoint free linear system owing to
  Seidenberg's theorem \cite[Thm.~1.7.1]{BS95} and \cite[Lem.~5.17]{KM98}.

  The ``$⇒$'' direction is contained in Remark~\ref{rem:restriction}.  For the
  other direction, let $X°$ be the big open subset of $X_{\reg}$ consisting of
  points where $ℰ_{X_{\reg}}$ is locally free, and let $\aF°$ be the family of
  locally free Higgs sheaves $(ℱ_{X°}, θ_{ℱ_{X°}})$ obtained by restricting
  Higgs sheaves $(ℱ_{X_{\reg}}, θ_{ℱ_{X_{\reg}}}) ∈ \TPIH_{X_{\reg}}$ to $X°$.
  From Corollary~\ref{cor:p2} above we obtain that $\aF°$ is bounded (as a
  family of Higgs sheaves).  Let $m ≫ 0$ be large enough and $D ∈ |m H|$ be
  general enough for both Corollary~\ref{cor:g1} and \cite[Prop.~7.3]{GKPT17} to
  apply to $D$ (as the first member of a tuple) with respect to $\aF°$ and
  $(ℰ_{X°}, θ_{ℰ_{X°}}) := (ℰ_{X_{\reg}}, θ_{ℰ_{X_{\reg}}})|_{X°}$.  By
  assumption, $(ℰ_{X_{\reg}}, θ_{ℰ_{X_{\reg}}})|_{D_{\reg}} ∈ \TPIH_{D_{\reg}}$,
  and hence Corollary~\ref{cor:g1} implies that there exists a locally free
  Higgs sheaf $(ℱ_{X_{\reg}}, θ_{ℱ_{X_{\reg}}}) ∈ \TPIH_{X_{\reg}}$ whose
  restriction to $D_{\reg}$ is isomorphic to
  $(ℰ_{X_{\reg}}, θ_{ℰ_{X_{\reg}}})|_{D_{\reg}}$.  Choosing a smooth, projective
  curve $C ⊂ D ∩ X°$ as in \cite[Prop.~7.3]{GKPT17}, it clearly follows that
  $(ℱ_{X°}, θ_{ℱ_{X°}})|_C ≅ (ℰ_{X°}, θ_{ℰ_{X°}})|_C$, from which we infer in a
  first step using \emph{loc.\ cit.} that
  $(ℱ_{X°}, θ_{ℱ_{X°}}) ≅ (ℰ_{X°}, θ_{ℰ_{X°}})$ and in a second step using
  reflexivity of both $ℱ_{X_{\reg}}$ and $ℰ_{X_{\reg}}$ that
  $$
  (ℰ_{X_{\reg}}, θ_{ℰ_{X_{\reg}}}) ≅ (ℱ_{X_{\reg}}, θ_{ℱ_{X_{\reg}}}) ∈
  \TPIH_{X_{\reg}}.  \eqno\qedhere
  $$
\end{proof}

%
%
\svnid{$Id: 04-higgs.tex 1269 2019-02-27 07:58:09Z taji $}

\section{Higgs bundles and Higgs sheaves}\label{sect:QHiggs}
\approvals{Behrouz & yes \\ Daniel & yes \\ Stefan & yes \\ Thomas & yes}

To prepare for the proof of the uniformisation result in the Part~\ref{part:II}
of this paper and to fix notation, we briefly recall the notion of $ℚ$-varieties
in the specific case of surfaces; details concerning this notion can be found in
\cite[Part~I]{GKPT15}.  Section~\ref{ssec:stability} discusses stability notions
for Higgs sheaves that are defined only on the smooth locus of a projective klt
space.

\subsection{Higgs sheaves on $ℚ$-surfaces}
\approvals{Behrouz & yes \\ Daniel & yes \\ Stefan & yes \\ Thomas & yes}

For surfaces with quotient singularities, we show how a Higgs $ℚ$-sheaf can be
constructed from a Higgs sheaf that is defined on the smooth locus of the
underlying surface; the theory can certainly be developed to cover more general
cases, but we restrict ourselves to the material necessary for the arguments in
the proof of our main result.  For further material, the reader is referred to
\cite[Sect.~5.5]{GKPT15}, where the notion of Higgs $ℚ$-sheaves is defined in
general.

\begin{construction}[Quasi-étale $ℚ$-structure on a surface with quotient singularities]\label{construction:Q}
  Let $S$ be a normal, projective surface with only quotient singularities.  By
  \cite[Prop.~3.10]{GKPT15}, the surface $S$ can then be equipped with the
  structure of a quasi-étale $ℚ$-variety in the sense of \cite[Defs.~3.1 and
  3.2]{GKPT15}.  Choosing one such structure, say $S^ℚ$, we are given a finite
  set $A$ and for each $α ∈ A$ a smooth, quasi-projective variety $S_α$ and a
  diagram
  \begin{equation}\label{eq:vxbf}
    \xymatrix{ %
      S_α \ar[rrr]_{\text{Galois with group }G_α} \ar@/^.4cm/[rrrrrr]^{p_α\text{, quasi-étale}} &&& U_α \ar[rrr]_{p'_α\text{, étale}} &&& S,
    }
  \end{equation}
  such that images $p(S_α)$ cover $S$, and such that certain compatibility
  conditions hold; we refer the reader to \cite[Sect.~3.3]{GKPT15} for a further
  discussion.

  Next, recall from Mumford in \cite[Sect.~2]{MR717614} or
  \cite[Sect.~3.4]{GKPT15} that $S^ℚ$ admits a \emph{global, Cohen-Macaulay
    cover}.  In other words, there exists a finite, Galois morphism
  $γ: \what{S} → S$ from a normal (and hence Cohen-Macaulay) surface, with
  Galois group $G:=\Gal(\what{S}/ S)$, and for every $α ∈ A$ a commutative
  diagram as follows,
  $$
  \xymatrix{
    \what{S}_α \ar[rrrr]^{q_α}_{\text{Galois with group }H_α \triangleleft G} \ar@{^(->}[d]_{\text{incl.\ of open}} &&&& S_α \ar[rrrr]_{\text{Galois with group }G_α = G/H_α} &&&& U_α \ar[d]^{p'_α\text{, étale}} \\
    \what{S} \ar[rrrrrrrr]_{γ\text{, Galois with group }G} &&&&&&&& S.
  }
  $$
\end{construction}

\begin{construction}[Higgs $ℚ$-bundle from Higgs bundle on $S_{\reg}$]\label{cons:Q-shf}
  In the setup of Construction~\ref{construction:Q}, assume that $S_{\reg}$ is
  equipped with a locally free Higgs sheaf $(ℰ_{S_{\reg}}, θ_{ℰ_{S_{\reg}}})$.
  We will denote the reflexive extension of $ℰ_{S_{\reg}}$ to $S$ by $ℰ_S$.
  Slightly generalising \cite[Constructions~3.8 and 5.15]{GKPT15}, one
  constructs a locally free Higgs $ℚ$-sheaf $(ℰ_S, θ_{ℰ_S})^{[ℚ]}$ on $S^ℚ$,
  given by the collection $(ℰ_{S_α}, θ_{ℰ_{S_α}})$ of locally free Higgs sheaves
  on the $S_α$ that can be obtained by extending the Higgs bundles
  $(p_α|_{p_α^{-1}(S_{\reg})})^{*}(ℰ_{S_{\reg}}, θ_{ℰ_{S_{\reg}}})$ to a Higgs
  bundle on $S_α$ using the Riemann Extension Theorem.
  
  More precisely, we set $S_α° = S_α∖ p_α^{-1}(S_{\sing})$.  Each
  $ℰ_{S_α}° : = \big( p_α|_{p_α^{-1}(S_{\reg})} \big)^*(ℰ)$ extends to a locally
  free sheaf $ℰ_{S_α}$ on $S_α$.  In addition, as the induced Higgs fields on
  $ℰ_{S_α}°$ are sections of $\sEnd(ℰ_{S_α}°) ⊗ Ω¹_{S_α°}$, they extend to
  sections of $\sEnd(ℰ_{S_α}) ⊗ Ω¹_{S_α} $.  We denote these extended sections
  by $θ_{ℰ_{S_α}}$.  Now, the pull-backs $q_α^*(ℰ_{S_α}, θ_{ℰ_{S_α}})$ are
  locally free Higgs sheaves on $\what{S}_α$ that glue to give a locally free
  Higgs $G$-sheaf $(ℰ_{\what{S}}, θ_{ℰ_{\what{S}}})$ on $\what{S}$.  The
  construction works without change for sheaves without (or with the trivial)
  Higgs field.
\end{construction}

\begin{rem}\label{fact:u1}
  In Construction~\ref{cons:Q-shf}, we have $ℰ_{\what{S}} ≅ γ^{[*]} ℰ_S$ and
  therefore $γ_*(ℰ_{\what{S}})^G ≅ ℰ_S$.
\end{rem}

\begin{construction}[Pull-back out of a $ℚ$-structure]\label{cons:pbfq}
  In the setting of Construction~\ref{cons:Q-shf}, consider a $G$-equivariant
  resolution of singularities, $π: \wtilde{S} → \what{S}$.  Given an arbitrary
  Higgs sheaf $(ℱ,τ)$ on $\what{S}$, there is generally no way to define a Higgs
  field on the pull-back $π^* ℱ$ even in cases where $ℱ$ is locally free.  In
  our special setting, however, recall from \cite[Lem.~5.17]{GKPT15} that there
  exists a $G$-invariant Higgs field on the $G$-sheaf
  $ℰ_{\wtilde{S}} := π^* ℰ_{\what{S}}$ that agrees with the pull-back of
  $θ_{ℰ_{\what{S}}}$ wherever $\what{S}$ is smooth.
\end{construction}

\subsection{Stability and polystability}
\approvals{Behrouz & yes \\ Daniel & yes \\ Stefan & yes \\ Thomas & yes}
\label{ssec:stability}

Stability properties of Higgs sheaves on singular, projective varieties are
defined and discussed in detail in \cite[Sect.~5.6]{GKPT15}.  In the situation
at hand, it makes sense to generalise this to the case where a Higgs sheaf is
defined on the smooth locus only.  We recall \cite[Def.~2.19]{GKPT17} in our
setting.

\begin{defn}[Stability for Higgs sheaves on the smooth locus]\label{def:sHsl}
  Let $X$ be a normal, projective variety, let $H ∈ \Div(X)$ be nef and let
  $(ℰ_{X_{\reg}},θ_{ℰ_{X_{\reg}}})$ be a torsion free Higgs sheaf on $X_{\reg}$.
  We say that $(ℰ_{X_{\reg}},θ_{ℰ_{X_{\reg}}})$ is \emph{stable with respect to
    $H$} if any generically Higgs-invariant subsheaf
  $ℱ_{X_{\reg}} ⊊ ℰ_{X_{\reg}}$ with
  $0 < \rank ℱ_{X_{\reg}} < \rank ℰ_{X_{\reg}}$ satisfies
  $$
  \frac{c_1 \bigl( ι_* ℱ_{X_{\reg}} \bigr)·[H]}{\rank ℱ_{X_{\reg}}} < \frac{c_1
    \bigl( ι_* ℰ_{X_{\reg}} \bigr)·[H]}{\rank ℰ_{X_{\reg}}},
  $$
  where $ι : X_{\reg} → X$ is the inclusion.  Analogously, we define
  \emph{semistable} and \emph{polystable}.
\end{defn}

The stability notion of Definition~\ref{def:sHsl} is compatible with the
existing notions.  The proof of the following fact is elementary and therefore
omitted.

\begin{fact}[\protect{Compatibility with existing notions, \cite[Lem.~2.21]{GKPT17}}]\label{lem:HQ2}
  Let $X$ be a normal, projective variety, let $H ∈ \Div(X)$ be nef, and let
  $(ℰ_X,θ_{ℰ_X})$ be a reflexive Higgs sheaf that is defined on all of $X$.
  Then, the following are equivalent.
  \begin{enumerate}
  \item The sheaf $(ℰ_X, θ_{ℰ_X})$ is stable with respect to $H$.
  \item The sheaf $(ℰ_X, θ_{ℰ_X})|_{X_{\reg}}$ is stable with respect to $H$.
  \end{enumerate}
  The analogous statements hold in the semistable and polystable setting.  \qed
\end{fact}

The following lemma discusses behaviour of stability when taking tensor
products, and shows polystability for Higgs bundles in $\TPIH$.

\begin{lem}[Polystability in tensor products and $\TPIH$]\label{lem:stvbn}
  Let $X$ be a normal, projective variety, let $H$ be an ample divisor on $X$,
  and let $(ℱ_{X_{\reg}}, θ_{ℱ_{X_{\reg}}})$ be a Higgs bundle on $X_{\reg}$.
  Assume that one of the following holds.
  \begin{enumerate}
  \item The Higgs bundle $(ℱ_{X_{\reg}}, θ_{ℱ_{X_{\reg}}})$ is in
    $\TPIH_{X_{\reg}}$.
  \item The Higgs bundle $(ℱ_{X_{\reg}}, θ_{ℱ_{X_{\reg}}})$ is a tensor product
    of two Higgs bundles on $X_{\reg}$ that are stable with respect to $H$.
  \end{enumerate}
  Then, $(ℱ_{X_{\reg}}, θ_{ℱ_{X_{\reg}}})$ is polystable with respect to $H$.
  In particular, if $(ℰ_{X_{\reg}},θ_{ℰ_{X_{\reg}}})$ is a Higgs bundle on
  $X_{\reg}$ that is stable with respect to $H$, then
  $\sEnd(ℰ_{X_{\reg}}, θ_{ℰ_{X_{\reg}}})$ is polystable with respect to $H$.
\end{lem}
\begin{proof}
  To begin, we remark that $(ℱ_{X_{\reg}}, θ_{ℱ_{X_{\reg}}})$ is semistable.  In
  fact, Remark~\ref{rem:restriction} and the restriction theorem for semistable
  Higgs sheaves on $X_{\reg}$, \cite[Thm.~6.1]{GKPT17}, respectively, allows us
  to restrict ourselves to the case where $X = X_{\reg}$ is a smooth projective
  curve.  There, the result is classically known in either of the two cases.

  The proof of polystability proceeds by induction on the dimension on $X$.  If
  $\dim X = 1$, then $X$ is a smooth projective curve and the result is
  classically known in either case.  As for the inductive step, assume that the
  result was known for all varieties of dimension less than $\dim X$, and assume
  that there exists a saturated, Higgs-invariant subsheaf
  $𝒜_{X_{\reg}} ⊊ ℱ_{X_{\reg}}$ whose slope equals that of $ℱ_{X_{\reg}}$ and
  that is stable with respect to $H$.  We need to show that $𝒜_{X_{\reg}}$ is a
  direct summand.  More precisely, we need to find a morphism of Higgs sheaves,
  $ℱ_{X_{\reg}} → 𝒜_{X_{\reg}}$, that is a projection onto $𝒜_{X_{\reg}}$.

  Now, if $m ≫ 0$ is sufficiently large and $Y ∈ |m·H|$ is sufficiently general,
  then $Y$ is normal, $Y_{\reg} = Y ∩ X_{\reg}$ and the following will hold.
  \begin{enumerate}
  \item\label{il:4-31-1} The subsheaf
    $𝒜_{X_{\reg}}|_{Y_{\reg}} ⊊ ℱ_{X_{\reg}}|_{Y_{\reg}}$ Higgs-invariant and
    stable.  This is the Restriction Theorem for stable Higgs sheaves on
    $X_{\reg}$, \cite[Thm.~6.1]{GKPT17}.

  \item\label{il:4-31-2} The restriction
    $\Hom(ℱ_{X_{\reg}}, 𝒜_{X_{\reg}} ) → \Hom(ℱ_{X_{\reg}}|_{Y_{\reg}},
    𝒜_{X_{\reg}}|_{Y_{\reg}} )$ is bijective.  This follows from standard
    arguments, compare \cite[Prop.~7.3 and proof]{GKPT17}.
  \end{enumerate}
  Item~\ref{il:4-31-1} together with the induction hypothesis show that there
  exists a projection map $ℱ_{X_{\reg}}|_{Y_{\reg}} → 𝒜_{X_{\reg}}|_{Y_{\reg}}$
  respecting Higgs fields.  Item~\ref{il:4-31-2} implies that this map extends
  to a projection $ℱ_{X_{\reg}} → 𝒜_{X_{\reg}}$, presenting $𝒜_{X_{\reg}}$ as a
  direct summand of $ℱ_{X_{\reg}}$.  A variant of \ref{il:4-31-2} that we leave
  to the reader, using morphisms of Higgs-sheaves instead of sheaf morphisms,
  shows that for $m$ sufficiently large, the projection
  $ℱ_{X_{\reg}} → 𝒜_{X_{\reg}}$ is in fact a morphism of Higgs sheaves,
  cf.~\cite[Lem.~5]{MR1157844} and \cite[Lem.~3.4]{MR2310103}.
\end{proof}

%
%
\svnid{$Id: 05-harmonicExistence.tex 1266 2019-02-26 19:07:11Z greb $}

\section{Existence of harmonic structures}
\approvals{Behrouz & yes \\ Daniel & yes \\ Stefan & yes \\ Thomas & yes}%
\label{sec:flatness}

The following is the main result in Part~\ref{part:0} of the present paper.  It
generalises earlier results obtained in \cite{MR84, SBW94, LT14, GKP13}.

\begin{thm}[Existence of harmonic structures]\label{thm:Qflat1new}
  Let $X$ be a projective klt space of dimension $n ≥ 2$.  Let $H ∈ \Div(X)$ be
  ample and use $H$ to equip $X_{\reg}$ with a Kähler metric.  Let
  $(ℰ_{X_{\reg}}, θ_{ℰ_{X_{\reg}}})$ be a reflexive Higgs sheaf on $X_{\reg}$.
  Denote the reflexive extension of $ℰ_{X_{\reg}}$ to $X$ by $ℰ_X$.  Then, the
  following statements are equivalent.
  \begin{enumerate}
  \item\label{il:5-1-1} The Higgs sheaf $(ℰ_{X_{\reg}}, θ_{ℰ_{X_{\reg}}})$ is
    (poly)stable with respect to $H$ and the $ℚ$-Chern characters satisfy
    \begin{equation*}
      \what{ch}_1(ℰ_X)·[H]^{n-1} = 0 \quad\text{and}\quad \what{ch}_2(ℰ_X)·[H]^{n
        -2} = 0.
    \end{equation*}
    
  \item\label{il:5-1-2} The sheaf $ℰ_{X_{\reg}}$ is locally free and
    $(ℰ_{X_{\reg}}, θ_{ℰ_{X_{\reg}}})$ is induced by a tame, purely imaginary
    harmonic bundle whose associated flat bundle is (semi)simple.
  \end{enumerate}
\end{thm}

\subsection{Proof of Theorem~\ref*{thm:Qflat1new}}
\approvals{Behrouz & yes \\ Daniel & yes \\ Stefan & yes \\ Thomas & yes}%
\label{ssec:poticwXs}

The two implications are proven separately.

\subsection*{Implication \ref{il:5-1-1} $⇒$ \ref{il:5-1-2}}
\label{subsubsect:first_direction}

Suppose we already know that $(ℰ_{X_{\reg}}, θ_{ℰ_{X_{\reg}}})$ admits a tame
and purely imaginary harmonic bundle structure $𝔼_{X_{\reg}}$.  Then the induced
flat bundle $(E_{X_{\reg}}, ∇_{𝔼_{X_{\reg}}})$ is always semisimple, see
Remark~\ref{rem:tame_semisimple}.  If $(ℰ_{X_{\reg}}, θ_{ℰ_{X_{\reg}}})$ is even
stable, then Lemma~\ref{lem:flshtpihb} implies that
$(E_{X_{\reg}}, ∇_{𝔼_{X_{\reg}}})$ is simple.  It therefore remains to establish
a tame, purely imaginary harmonic bundle structure.  The proof is rather long,
and therefore subdivided into six steps.

\subsection*{Implication \ref{il:5-1-1} $⇒$ \ref{il:5-1-2}, Step~1: Reduction to stable sheaves on surfaces}
\approvals{Behrouz & yes \\ Daniel & yes \\ Stefan & yes \\ Thomas & yes}%

Standard arguments involving the Bogomolov-Gieseker inequality show that it
suffices to consider the stable case only.  We refer to \cite[Step~2 in proof of
Thm.~6.2]{GKP15} for details.

\begin{asswlog}\label{asswlog:1}
  The Higgs sheaf $(ℰ_{X_{\reg}}, θ_{ℰ_{X_{\reg}}})$ is stable with respect to
  $H$.
\end{asswlog}
\CounterStep

Choosing a sufficiently increasing sequence of numbers
$0 ≪ m_1 ≪ m_2 ≪ ⋯ ≪ m_{n-2}$ as well as a sufficiently general tuple of
hyperplanes,
$$
(D_1, …, D_{n-2}) ∈ |m_1·H| ⨯ ⋯ ⨯ |m_{n-2}·H|,
$$
the following will hold.
\begin{enumerate}
\item\label{il:t1} The intersection $S := D_1 ∩ ⋯ ∩ D_{n-2}$ is irreducible and
  normal.  Moreover, $S$ is a klt space and $S_{\reg} = S ∩ X_{\reg}$.  This is
  Seidenberg's theorem \cite[Thm.~1.7.1]{BS95} and \cite[Lem.~5.17]{KM98}.
  
\item\label{il:t2} The restricted sheaf $ℰ_S := ℰ_X|_S$ is reflexive, hence
  locally free on $S_{\reg}$, and satisfies $\what{ch}_2(ℰ_S) = 0$ as well as
  $\what{c}_1(ℰ_Y)·[H|_S] = 0$.  This is \cite[Cor.~1.1.14]{MR2665168} and
  \cite[Thm.~3.13]{GKPT15}.

\item\label{il:t3} The locally free Higgs sheaf
  $(ℰ_{S_{\reg}}, θ_{ℰ_{S_{\reg}}}) := (ℰ_{X_{\reg}},
  θ_{ℰ_{X_{\reg}}})|_{S_{\reg}}$ is stable with respect to $H|_S$.  This is the
  Restriction Theorem \cite[Thm.~6.1]{GKPT17}.
  
\item\label{il:t4} The reflexive Higgs sheaf $(ℰ_{X_{\reg}}, θ_{ℰ_{X_{\reg}}})$
  is locally free and admits a tame and purely imaginary harmonic bundle
  structure if and only if the locally free Higgs sheaf
  $(ℰ_{X_{\reg}}, θ_{ℰ_{X_{\reg}}})|_{S_{\reg}}$ admits a tame and purely
  imaginary harmonic bundle structure.  This is Corollary~\ref{cor:restr}.
\end{enumerate}

Items~\ref{il:t1}--\ref{il:t3} imply that the surface $S$ and the Higgs sheaf
$(ℰ_{X_{\reg}}, θ_{ℰ_{X_{\reg}}})|_{S_{\reg}}$ reproduce the assumptions made in
Theorem~\ref{thm:Qflat1new}.  Item~\ref{il:t4} implies that it suffices to show
that the locally free Higgs sheaf $(ℰ_{X_{\reg}}, θ_{ℰ_{X_{\reg}}})|_{S_{\reg}}$
admits a tame and purely imaginary harmonic bundle structure.  Replacing $X$ by
$S$ we can (and will) therefore assume the following.

\begin{asswlog}\label{asswlog:dim2}
  The dimension of $X$ is equal to two.
\end{asswlog}

\subsection*{Implication \ref{il:5-1-1} $⇒$ \ref{il:5-1-2}, Step~2: The $ℚ$-structure}
\approvals{Behrouz & yes \\ Daniel & yes \\ Stefan & yes \\ Thomas & yes}%

As we have seen in \ref{il:t2} above, assumption~\ref{asswlog:dim2} immediately
implies that the reflexive sheaf $ℰ_{X_{\reg}}$ is locally free.  In a similar
vein, since $X$ is a klt \emph{surface}, we obtain immediately that $X$ has
quotient singularities.  We may therefore equip $X$ with a quasi-étale
$ℚ$-structure and choose a global Cohen-Macaulay cover with Galois group $G$
allowing to compute $ℚ$-Chern classes as in \cite[Thm.~3.13]{GKPT15}, and a
$G$-equivariant, strong log-resolution of singularities as follows,
$$
\xymatrix{
  \wtilde{X} \ar[rrrr]_{\txt{\scriptsize $π$, $G$-equivariant, strong\\\scriptsize log-resolution of singularities}} \ar@/^4mm/[rrrrrrrr]^{ψ} &&&&
  \what{X} \ar[rrrr]_{\txt{\scriptsize $γ$, global Cohen-Macaulay cover}} &&&& X.
}
$$
The locally free Higgs $G$-bundle on $\wtilde{X}$ obtained in
Construction~\ref{cons:pbfq} will be denoted by
$(ℰ_{\wtilde{X}}, θ_{ℰ_{\wtilde{X}}})$.  The following immediate consequence of
Assumption~\ref{asswlog:1} will be used later.

\begin{claim}[Stability and Chern classes of $ℰ_{\wtilde{X}}$]\label{claim:ccl1}
  The Higgs $G$-bundle $(ℰ_{\wtilde{X}}, θ_{ℰ_{\wtilde{X}}})$ is $G$-stable with
  respect to a $G$-invariant ample divisor $\wtilde{H} ∈ \Div(\wtilde{X})$.  The
  Chern classes $c_i(ℰ_{\wtilde{X}}) ∈ H^{2i} \bigl(\wtilde{X},\, ℝ\bigr)$
  vanish.
\end{claim}
\begin{proof}[Proof of Claim~\ref{claim:ccl1}]
  The first claim is proven as in the proof of \cite[Prop.~6.2]{GKPT15}, using
  the locally free Higgs $G$-sheaf $(ℰ_{\what{X}}, θ_{ℰ_{\what{X}}})$ on
  $\what{X}$ obtained in Construction~\ref{cons:Q-shf}.  For the second claim,
  observe that the maximally destabilising subsheaf of the Higgs
  bundle\footnote{Here, we view $(ℰ_{\wtilde{X}}, θ_{ℰ_{\wtilde{X}}})$ as a
    Higgs-bundle without its structure as a $G$-sheaf}
  $(ℰ_{\wtilde{X}}, θ_{ℰ_{\wtilde{X}}})$ with respect to $\wtilde{H}$ is
  automatically $G$-invariant.  In particular, it follows from $G$-stability
  that $(ℰ_{\wtilde{X}}, θ_{ℰ_{\wtilde{X}}})$ is at least semistable with
  respect to $\wtilde{H}$.  Moreover, from \ref{il:t2}, from
  \cite[Thm.~3.13]{GKPT15}, and from the functorial properties of Chern classes
  \cite[Thm.~3.2(d)]{Fulton98} we conclude that
  $c_1(ℰ_{\wtilde{X}}) · \wtilde{H} = ch_2(ℰ_{\wtilde{X}})= 0$.  Vanishing of
  the $c_i$ hence follows from the Hodge index theorem and the
  Bogomolov-Gieseker inequality as in the first paragraph of
  \cite[Sect.~6.3]{GKP15}.  \qedhere~(Claim~\ref{claim:ccl1})
\end{proof}

For the remainder of the proof, we fix one $G$-invariant ample divisor
$\wtilde{H} ∈ \Div(\wtilde{X})$ as in Claim~\ref{claim:ccl1} and use this
divisor to equip $\wtilde{X}$ with a $G$-invariant Kähler metric
$ω_{\wtilde H}$.  We denote the holomorphic vector bundle associated with
$ℰ_{X_{\reg}}$ by $(E_{X_{\reg}}, \bar{∂}_{E_{X_{\reg}}})$.  We use similar
notation also for other bundles, including
$(E_{\wtilde{X}}, \bar{∂}_{E_{\wtilde{X}}})$.

\subsubsection*{Open sets}
\approvals{Behrouz & yes \\ Daniel & yes \\ Stefan & yes \\ Thomas & yes}%

Throughout the proof, we consider the big, open subset
$X° := X_{\reg} ∖ γ(\what{X}_{\sing})$ of $X$.  Preimages and restricted
morphisms are written as
$$
\xymatrix{ %
  \wtilde{X}° \ar[rrr]_{π°\text{, isomorphism}} \ar@/^4mm/[rrrrrr]^{ψ°\text{, Galois, with group $G$}} &&& \what{X}° \ar[rrr]_{γ°\text{, Galois, with group $G$}} &&& X°.
}
$$
Write $(ℰ_{X°}, θ_{X°}) := (ℰ_{X_{\reg}}, θ_{X_{\reg}})|_{X°}$.  We use similar
notation also for other open sets and write
$(ℰ_{\wtilde{X}°}, θ_{ℰ_{\wtilde{X}°}})$.

\subsubsection*{Orbifold charts}
\approvals{Behrouz & yes \\ Daniel & yes \\ Stefan & yes \\ Thomas & yes}%

We also consider orbifold charts, as introduced in
Construction~\ref{construction:Q}.  The relevant morphisms and their
restrictions are summarised in the following diagram,
$$
\xymatrix{ %
  \wtilde{X}_α \ar[rrr]_{π_α\text{, resolution of sings.}} \ar@/^5mm/[rrrrrr]^{Π_α} &&& \what{X}_α \ar[rrr]_{q_α\text{, Galois, with group $H_α$}} &&& X_α \ar[rr]^{\text{quasi-étale}} && X \\
  \wtilde{X}°_α \ar[rrr]^{π°_α\text{, isomorphism}} \ar@/_5mm/[rrrrrr]_{Π°_α\text{, Galois, with group $H_α$}} \ar@{^(->}[u] &&& \wtilde{X}°_α \ar[rrr]^{q°_α\text{, Galois, with group $H_α$}} \ar@{^(->}[u] &&& X°_α \ar[rr]_{\text{étale}} \ar@{^(->}[u] && X°.  \ar@{^(->}[u]
}
$$

\subsection*{Implication \ref{il:5-1-1} $⇒$ \ref{il:5-1-2}, Step~3: Simplification}
\approvals{Behrouz & yes \\ Daniel & yes \\ Stefan & yes \\ Thomas & yes}%

The following claim allows us to concentrate on the big, open set
$X° ⊆ X_{\reg}$, simplifying notation substantially.

\begin{claim}\label{claim:simpla}
  To prove Theorem~\ref{thm:Qflat1new}, it suffices to show that the Higgs
  bundle $(ℰ_{X°}, θ_{ℰ_{X°}})$ admits a tame and purely imaginary harmonic
  bundle structure.
\end{claim}
\Publication{
  \begin{proof}[Proof of Claim~\ref{claim:simpla}]
    This is a direct consequence of Fact~\ref{fact:flatext} and
    Corollary~\ref{cor:uniq}.  A detailed proof is found in the preprint version
    of this paper.  \qedhere\ (Claim~\ref{claim:simpla})
  \end{proof} %
} 
\Preprint{
  \begin{proof}[Proof of Claim~\ref{claim:simpla}]
    Assume that $(E_{X°}, \bar{∂}_{E_{X°}})$ admits a tame and purely imaginary
    harmonic bundle structure, say $𝔼_{X°}$.  Since $X°$ is a big subset of
    $X_{\reg}$, the natural morphisms between fundamental groups,
    $π_1(X°) → π_1(X_{\reg})$ is isomorphic.  As we recalled in
    Fact~\ref{fact:flatext}, this implies that the semisimple flat bundle
    $(E_{X°}, ∇_{𝔼_{X°}})$ on $X°$ extends to a semisimple flat bundle
    $(E'_{X_{\reg}}, ∇')$ on $X_{\reg}$.  The Jost-Zuo existence result for
    harmonic structures, Theorem~\ref{thm:JZM}, therefore applies and equips
    $E'_{X_{\reg}}$ with the structure of a tame and purely imaginary harmonic
    bundle $𝔼'_{X_{\reg}}$ with connection $∇_{𝔼'_{X_{\reg}}} = ∇'$.  Write
    $(ℰ'_{X_{\reg}}, θ'_{ℰ'_{X_{\reg}}})$ for the associated locally free Higgs
    sheaf.  The restriction of $𝔼'_{X_{\reg}}$ to $X°$ is tame and purely
    imaginary with connection $∇_{𝔼_{X°}}$.  Corollary~\ref{cor:uniq} hence
    implies
    $$
    (ℰ'_{X_{\reg}}, θ_{ℰ'_{X_{\reg}}})|_{X°} ≅ (ℰ_{X°}, θ_{ℰ_{X°}}) =
    (ℰ_{X_{\reg}}, θ_{ℰ_{X_{\reg}}})|_{X°}.
    $$
    Since $X°$ is a big subset of $X_{\reg}$, we conclude that
    $(ℰ_{X_{\reg}}, θ_{ℰ_{X_{\reg}}})$ admits a tame and purely imaginary
    harmonic bundle structure.  \qedhere \ (Claim~\ref{claim:simpla})
  \end{proof} %
} 

\subsection*{Implication \ref{il:5-1-1} $⇒$ \ref{il:5-1-2}, Step~4: $G$-invariant harmonic structure on $E_{\wtilde{X}}$}
\approvals{Behrouz & yes \\ Daniel & yes \\ Stefan & yes \\ Thomas & yes}

To start the core of the argument, we will now show the existence of a
$G$-invariant harmonic structure on $(ℰ_{\wtilde{X}}, θ_{ℰ_{\wtilde{X}}})$.

\begin{claim}\label{claim:xza1}
  The locally free Higgs $G$-sheaf $(ℰ_{\wtilde{X}}, θ_{ℰ_{\wtilde{X}}})$ admits
  a $G$-invariant harmonic bundle structure
  $𝔼_{\wtilde{X}} := (E_{\wtilde{X}}, \bar{∂}_{E_{\wtilde{X}}},
  θ_{ℰ_{\wtilde{X}}}, h_{E_{\wtilde{X}}})$.
\end{claim}
\begin{proof}[Proof of Claim~\ref{claim:xza1}]
  The $G$-stability found in Claim~\ref{claim:ccl1} allows us to apply a
  classical result of Simpson, \cite[Thm.~1 on p.~878]{MR944577}, which states
  that that $E_{\wtilde{X}}$ carries a $G$-invariant Hermitian metric
  $h_{E_{\wtilde{X}}}$ that is Hermitian Yang-Mills with respect to the Kähler
  metric $ω_{\wtilde{H}}$.  Using the notation of Fact and
  Definition~\ref{def:harm}, vanishing of Chern classes, Claim~\ref{claim:ccl1},
  and \cite[discussion on p.~16f]{MR1179076} imply that the connection
  $$
  ∇_{𝔼_{\wtilde{X}}} := ∂_{E_{\wtilde{X}}} + \bar{∂}_{E_{\wtilde{X}}} +
  θ_{ℰ_{\wtilde{X}}} + θ_{ℰ_{\wtilde{X}}}^{h_{E_{\wtilde{X}}}}
  $$
  is in fact flat and hence that
  $(E_{\wtilde{X}}, \bar{∂}_{E_{\wtilde{X}}}, θ_{ℰ_{\wtilde{X}}},
  h_{E_{\wtilde{X}}})$ is a harmonic bundle.  \qedhere\ (Claim~\ref{claim:xza1})
\end{proof}

\begin{claim}\label{claim:kl1}
  If $F_{\wtilde{X}}$ is a $G$-invariant flat subbundle of $E_{\wtilde{X}}$,
  then either $F_{\wtilde{X}} = 0$ or $F_{\wtilde{X}} = E_{\wtilde{X}}$.
\end{claim}
\begin{proof}[Proof of Claim~\ref{claim:kl1}]
  We have seen in Lemma~\ref{lem:flshtpihb} that $F_{\wtilde{X}}$ is a
  holomorphic subbundle of $(E_{\wtilde{X}}, \bar{∂}_{E_{\wtilde{X}}})$ and
  yields a Higgs-invariant, locally free subsheaf $ℱ_{\wtilde{X}}$ of
  $(ℰ_{\wtilde{X}}, θ_{ℰ_{\wtilde{X}}})$.  The assumption that $F_{\wtilde{X}}$
  is a $G$-invariant subbundle of $E_{\wtilde{X}}$ guarantees that
  $ℱ_{\wtilde{X}}$ is a Higgs-invariant $G$-subsheaf of the Higgs $G$-sheaf
  $(ℰ_{\wtilde{X}}, θ_{ℰ_{\wtilde{X}}})$.  Using that $F_{\wtilde{X}}$ is
  invariant with respect to the flat connection $∇_{𝔼_{\wtilde{X}}}$, we obtain
  that $F_{\wtilde{X}}$ is again flat, so that all its Chern classes vanish.
  Claim~\ref{claim:kl1} thus follows from Claim~\ref{claim:ccl1} above.
  \qedhere\ (Claim~\ref{claim:kl1})
\end{proof}

\begin{claim}\label{claim:kl2}
  If $F_{\wtilde{X}°}$ is a $G$-invariant flat subbundle of
  $(E_{\wtilde{X}°}, ∇_{𝔼_{\wtilde{X}}}|_{\wtilde{X}°})$, then $F_{\wtilde{X}°}$
  extends to a $G$-invariant flat subbundle of $E_{\wtilde{X}}$.
\end{claim}
\begin{proof}[Proof of Claim~\ref{claim:kl2}]
  The question is local over the $X_α$.  More precisely, using the notation
  introduced in Step~2, it suffices to show that for every index $α$, the
  restricted bundle $F_{\wtilde{X}°_α} := F_{\wtilde{X}°}|_{\wtilde{X}°_α}$ on
  $\wtilde{X}°_α$ extends to a subbundle $F_{\wtilde{X}_α} ⊆ E_{\wtilde{X}_α}$
  on $\wtilde{X}_α$ that is invariant with respect to $∇_{𝔼_{\wtilde{X}}}$.  The
  $G$-invariance follows then automatically from density of
  $\wtilde{X}°_α ⊆ \wtilde{X}_α$.
 
  Recall from Constructions~\ref{cons:Q-shf} and \ref{cons:pbfq} that
  $E_{\wtilde{X}°_α}$ is a pull-back from $X°_α$, say
  $E_{\wtilde{X}°_α} ≅ (Π°_α)^* E_{X°_α}$.  We claim that both the connection
  $∇_{𝔼_{\wtilde{X}}}|_{\wtilde{X}°}$ and the subbundle $F_{\wtilde{X}°_α}$
  descend to $X°_α$, too.  Indeed, since $∇_{𝔼_{\wtilde{X}}}$ is invariant under
  $G$, and hence also invariant under the Galois group
  $H_α = \operatorname{Gal}(Π°_α) ⊆ G$, we may apply
  Proposition~\ref{prop:conndescent} to show the existence of a connection
  $∇_{E_{X°_α}}$ on $E_{X°_α}$ such that
  $∇_{𝔼_{\wtilde{X}}}|_{\wtilde{X}°} = (Π°_α)^* ∇_{E_{X°_α}}$.  Moreover, it
  follows from \cite[Prop.~2.16]{GKPT15} applied to the associated locally free
  sheaves of $𝒪_{\wtilde{X}°_α}$-modules that there exists a subbundle
  $F_{X°_α} ⊆ E_{X°_α}$ with $F_{\wtilde{X}°_α} = (Π°_α)^* F_{X°_α}$.  The
  subbundle $F_{X°_α}$ is then clearly invariant with respect to the connection
  $∇_{E_{X°_α}}$.

  As $X°_α$ is a big open subset of the smooth, quasi-projective surface $X_α$,
  the natural morphism of fundamental groups, $π_1(X°_α) → π_1(X_α)$, is
  isomorphic.  Fact~\ref{fact:flatext} therefore asserts that the subbundle
  $F_{X°_α} ⊆ E_{X°_α}$ extends from $X°_α$ to a $∇_{E_{X°_α}}$-invariant
  subbundle $F_{X°_α} ⊆ E_{X°_α}$ that exists on all of $X_α$.  Pulling back, we
  define the desired extension of $F_{\wtilde{X}°_α}$ as
  $F_{\wtilde{X}_α} := (Π°_α)^* F_{X°_α}$.  \qedhere\ (Claim~\ref{claim:kl2})
\end{proof}

Combining Claims~\ref{claim:kl1} and \ref{claim:kl2}, we arrive at the following
result.

\begin{cons}\label{cons:xp1}
  If $F_{\wtilde{X}°}$ is a $G$-invariant subbundle of $E_{\wtilde{X}°}$ that is
  also invariant with respect to $∇_{𝔼_{\wtilde{X}}}|_{\wtilde{X}°}$, then
  $F_{\wtilde{X}°} = 0$ or $F_{\wtilde{X}°} = E_{\wtilde{X}°}$.  \qed\
  (Consequence~\ref{cons:xp1})
\end{cons}

\subsection*{Implication \ref{il:5-1-1} $⇒$ \ref{il:5-1-2}, Step~5: Harmonic structure on $E_{X°}$}
\approvals{Behrouz & yes \\ Daniel & yes \\ Stefan & yes \\ Thomas & yes}

Eventually, we would like to show that the metric $h_{E_{\wtilde{X}}}$ descends
to a smooth Hermitian metric on $E_{X°}$.  However, owing to branching of the
map $γ$ over the smooth part of $X_{\reg}$, it is not clear from the outset
whether the natural stratified $\cC^∞$-structure on the quotient $\wtilde{X}°/G$
will coincide with the $\cC^∞$-structure induced by the complex structure on
$X°$.  Rather then showing descent of the metric directly, we will first discuss
the flat structure on $E°$ and construct a metric from there.

Write $∇_{E_{X°}}$ for the unique connection on $E_{X°}$ such that
$$
∇_{𝔼_{\wtilde{X}}}|_{\wtilde{X}°} = (ψ°)^*∇_{E_{X°}},
$$
which exists by Proposition~\ref{prop:conndescent}.  Consequence~\ref{cons:xp1}
implies that the flat bundle $(E_{X°}, ∇_{E_{X°}})$ is simple.  We may hence
apply the Jost-Zuo existence result for harmonic structures,
Theorem~\ref{thm:JZM}, to $(E_{X°}, ∇_{E_{X°}})$ and find a tame and purely
imaginary, harmonic bundle structure
$𝔼'_{X°} = (E_{X°}, \bar{∂}'_{E_{X°}}, θ'_{ℰ_{X°}}, h'_{E_{X°}})$ on $E_{X°}$
whose associated connection $∇_{𝔼'_{X°}}$ equals $∇_{E_{X°}}$.

\subsection*{Implication \ref{il:5-1-1} $⇒$ \ref{il:5-1-2}, Step~6: Comparison}
\approvals{Behrouz & yes \\ Daniel & yes \\ Stefan & yes \\ Thomas & yes}

There are now two tame and purely imaginary, $G$-invariant harmonic bundle
structures on $E_{\wtilde{X}°}$.  First, the restriction
$𝔼_{\wtilde{X}}|_{\wtilde{X}°}$, which is obviously tame and purely imaginary.
The associated connection $∇_{𝔼_{\wtilde{X}}}|_{\wtilde{X}°}$ on
$E_{\wtilde{X}°}$ is semisimple by Remark~\ref{rem:tame_semisimple}.  Second, by
\cite[Lem.~25.29]{MR2283665} the pull-back of the harmonic structure
$(ψ°)^* 𝔼'_{X°}$ on $E_{X°}$ is also tame and purely imaginary.  The associated
connection on $E_{\wtilde{X}°}$ is the pull-back of $∇_{E_{X°}}$, and hence
likewise equal to $∇_{𝔼_{\wtilde{X}}}|_{\wtilde{X}°}$.

Using the observation that the two tame and purely imaginary bundles induce the
same semisimple connection, Theorem~\ref{thm:JZM} now immediately implies that
the complex structures and Higgs fields on $\wtilde{X}°$ are equal
\begin{equation}\label{eq:vb}
  \begin{matrix}
    (ψ°)^*\bar{∂}_{E_{X°}} & = & \bar{∂}_{E_{\wtilde{X}}}|_{\wtilde{X}°} & = & (ψ°)^*\bar{∂}'_{E_{X°}} \\
    (ψ°)^* θ_{ℰ_{X°}} & = & θ_{ℰ_{\wtilde{X}}}|_{\wtilde{X}°} & = & (ψ°)^* θ'_{ℰ_{X°}}.
  \end{matrix}
\end{equation}
We emphasise that \eqref{eq:vb} is a statement about equalities of operators,
rather than mere isomorphisms of complex structures and Higgs bundles.  In
particular, \eqref{eq:vb} implies that the holomorphic structure
$\bar{∂}_{E_{X°}}$ corresponding to $ℰ_{X°}$ equals $\bar{∂}'_{E_{X°}}$ and
furthermore that $θ_{ℰ_{X°}} = θ'_{ℰ_{X°}}$ on $X°$.  Consequently,
$(ℰ_{X_{\reg}}, θ_{ℰ_{X_{\reg}}}) ∈ \TPIH_{X_{\reg}}$.  The implication
\ref{il:5-1-1} $⇒$ \ref{il:5-1-2} of Theorem~\ref{thm:Qflat1new} is thus
established.

\subsection*{Implication \ref{il:5-1-2} $⇒$ \ref{il:5-1-1}}
\label{subsubsect:other_way}
\approvals{Behrouz & yes \\ Daniel & yes \\ Stefan & yes \\ Thomas & yes}

First, remark that in the semisimple case the implication \ref{il:5-1-2} $⇒$
\ref{il:5-1-1} is immediate consequence of Lemma~\ref{lem:stvbn},
Proposition~\ref{prop:p1}, \cite[Thm.~3.10]{GKPT17}, and standard calculus of
$ℚ$-Chern classes; this is explained in detail in \cite[Sect.~3.8]{GKPT15}.

Second, assume in addition that
$(ℰ_{X_{\reg}}, θ_{ℰ_{X_{\reg}}}) ∈ \TPIH_{X_{\reg}}$ is induced by a harmonic
bundle $𝔼 = (E, \bar{∂}_{X_{\reg}}, θ_{ℰ_{X_{\reg}}}, h)$ whose associated flat
bundle $(E, ∇_𝔼)$ is simple, but suppose that
$(ℰ_{X_{\reg}}, θ_{ℰ_{X_{\reg}}})$ is not stable.  Being polystable, it splits
as a direct sum of stable bundles
$$
(ℰ_{X_{\reg}}, θ_{ℰ_{X_{\reg}}}) = (ℰ_{X_{\reg}}^{(1)}, θ^{(1)}_{ℰ_{X_{\reg}}}) ⊕ \dots ⊕ (ℰ_{X_{\reg}}^{(m)}, θ^{(m)}_{ℰ_{X_{\reg}}}) \quad \text{ for some }m ∈ ℕ,
$$
each of which satisfies the Chern character vanishings formulated in
Item~\ref{il:5-1-1} above, cf.\ \cite[Step~2 in proof of Thm.~6.2]{GKP15}.
Applying \cite[Thm.~6.1]{GKPT17} and the Lefschetz hyperplane theorem,
\cite[Thm.\ in Sect.~II.1.2]{GoreskyMacPherson}, we find a smooth complete
intersection curve $C ⊂ X_{\reg}$ such that each
$(ℰ_{X_{\reg}}^{(j)}, θ^{(j)}_{ℰ_{X_{\reg}}})|_C$ is stable and such that the
natural map
\begin{equation}\label{eq:surjectiveLefschetz}
  (ι_C)_*: π_1(C) \twoheadrightarrow π_1(X_{\reg})
\end{equation}
is surjective.  Recall from Remark~\ref{rem:restriction} that the Higgs bundle
\begin{equation}\label{eq:split_after_restriction}
  (ℰ_{X_{\reg}}, θ_{ℰ_{X_{\reg}}})|_C = (ℰ_{X_{\reg}}^{(1)}, θ^{(1)}_{ℰ_{X_{\reg}}})|_C ⊕ \dots ⊕ (ℰ_{X_{\reg}}^{(m)}, θ^{(m)}_{ℰ_{X_{\reg}}})|_C
\end{equation}
admits a harmonic bundle structure induced from $𝔼$ by restriction.  The
associated flat bundle corresponds to the representation obtained by composing
the monodromy representation of $(E, ∇_𝔼)$ with \eqref{eq:surjectiveLefschetz}
and is therefore simple.  However, via the Simpson correspondence on $C$,
\cite[Cor.~1.3]{MR1179076}, this yields a contradiction to the decomposition
\eqref{eq:split_after_restriction} of $(ℰ_{X_{\reg}}, θ_{ℰ_{X_{\reg}}})|_C$.
\qed

\part{Applications}
\label{part:II}

%
%
\svnid{$Id: 05a-nonabelian-klt.tex 883 2018-01-03 13:53:48Z peternell $}

\section{Nonabelian Hodge correspondences for smooth loci}
\subversionInfo
\approvals{Behrouz & yes \\ Daniel & yes \\ Stefan & yes\\ Thomas & yes}
\label{sec:klt}

The existence result for tame and purely imaginary harmonic bundles,
Theorem~\ref{thm:Qflat1new}, yields a nonabelian Hodge correspondence that
relates semisimple local systems on the smooth locus of a klt space to
polystable Higgs bundles on that locus.  As in Simpson's work, this
correspondence extends to a correspondence for arbitrary local systems.

\subsection{Nonabelian Hodge correspondence for polystable bundles}
\approvals{Behrouz & yes \\ Daniel & yes \\ Stefan & yes\\ Thomas & yes}

Before formulating the nonabelian Hodge correspondence for polystable bundles in
Theorem~\ref{thm:klt-smooth1} below, we need to specify the appropriate category
of bundles.  The following definition will be used.

\begin{defn}
  Let $X$ be a projective klt space $X$ and $ℰ_{X_{\reg}}$ be a locally free
  sheaf on $X_{\reg}$, whose extension to a reflexive sheaf on $X$ is denoted by
  $ℰ_X$.  If $H ∈ \Div(X)$ is ample, we say that \emph{$ℰ_{X_{\reg}}$ has
    vanishing $ℚ$-Chern classes with respect to $H$} if
  $\what{ch}_1(ℰ_X)·[H]^{n-1} = 0$ and $\what{ch}_2(ℰ_X)·[H]^{n-2} = 0$.
\end{defn}

As one immediate consequence of the existence result for harmonic structures,
Theorem~\ref{thm:Qflat1new}, we see that a Higgs bundle on the smooth locus of a
projective klt space is polystable and has vanishing $ℚ$-Chern classes after
cutting down with respect to one ample class, iff the same holds for any other
ample class.  This gives rise to the following fact, which we use to define the
relevant category of bundles.

\begin{factdef}[Category $\pHiggs_{X_{\reg}}$]
  Given a projective klt space $X$ and a locally free Higgs sheaf
  $(ℰ_{X_{\reg}}, θ_{ℰ_{X_{\reg}}})$ on $X_{\reg}$, the following conditions are
  equivalent.
  \begin{enumerate}
  \item There exists an ample $H ∈ \Div(X)$, such that
    $(ℰ_{X_{\reg}}, θ_{ℰ_{X_{\reg}}})$ is polystable and has vanishing $ℚ$-Chern
    classes with respect to $H$.

  \item For any ample $H ∈ \Div(X)$, the Higgs bundle
    $(ℰ_{X_{\reg}}, θ_{ℰ_{X_{\reg}}})$ is polystable and has vanishing $ℚ$-Chern
    classes with respect to $H$.
  \end{enumerate}
  With their natural morphisms, the Higgs bundle satisfying these conditions
  form a category, which we denote by $\pHiggs_{X_{\reg}}$.  \qed
\end{factdef}

The nonabelian Hodge correspondence for polystable bundles, which is a direct
analogue of \cite[Cor.~1.3]{MR1179076}, is now formulated as follows.

\begin{thm}[Nonabelian Hodge correspondence for $\pHiggs_{X_{\reg}}$]\label{thm:klt-smooth1}
  Let $X$ be a projective klt space.  Then, there exists an equivalence between
  the category $\pHiggs_{X_{\reg}}$ and the category $\sLSys_{X_{\reg}}$ of
  semisimple local systems on $X_{\reg}$.

  For $(ℰ_{X_{\reg}}, θ_{ℰ_{X_{\reg}}}) ∈ \pHiggs_{X_{\reg}}$ that are
  restrictions of polystable Higgs bundles on $X$ with vanishing Chern classes,
  and for local systems $\aE_{X_{\reg}} ∈ \sLSys_{X_{\reg}}$ that are
  restrictions of local systems on $X$, the correspondence is compatible with
  the global nonabelian Hodge Correspondence for projective klt spaces found in
  \cite[Sect.~3]{GKPT17}.
\end{thm}
\begin{proof}[Sketch of proof]
  Starting with a semisimple local system $\aE_{X_{\reg}}$ on
  $X_{\reg}$, let $(E_{X_{\reg}}, ∇_{E_{X_{\reg}}})$ be an associated flat
  bundle.  By Theorem~\ref{thm:JZM}, this bundle admits a tame and purely
  imaginary harmonic metric that is unique up to flat automorphisms; these
  preserve the induced decomposition \eqref{eq:Harm}.  The uniquely determined
  associated Higgs bundle $(ℰ_{X_{\reg}}, θ_{X_{\reg}})$ carries a unique
  algebraic structure by Remark and Notation \ref{rem:autoalg1a} and is moreover
  polystable by Lemma~\ref{lem:stvbn}.  Its $ℚ$-Chern classes vanish by
  Proposition~\ref{prop:p1}.  Assigning $(ℰ_{X_{\reg}}, θ_{X_{\reg}})$ to
  $(E_{X_{\reg}}, ∇_{E_{X_{\reg}}})$ defines a functor $η_{X_{\reg}}$ from
  $\sLSys_{X_{\reg}}$ to $\pHiggs_{X_{\reg}}$.  Compatibility with the global
  nonabelian Hodge Correspondence for projective klt spaces follows from the
  construction in \cite{GKPT17}; cf.~especially \cite[Prop.~3.10]{GKPT17}.

  We claim that the functor $η_{X_{\reg}}$ is an equivalence of categories; for
  that, we need to check that it is \emph{full}, \emph{faithful} and
  \emph{essentially surjective}.  While essential surjectivity quickly follows
  from Theorem~\ref{thm:Qflat1new}, we need to argue a bit more to establish the
  other two conditions.  To this end, let $γ: Y → X$ be a maximally quasi-étale
  cover, as provided by \cite[Thm.~1.5]{GKP13} and as used in the first part of
  Section~\ref{ssec:poticwXs0}.  Denote the Galois group of $γ$ by $G$.  By the
  defining property of the maximally quasi-étale cover, the pullback
  $γ^*(\aE_{X_{\reg}})$ extends to a $G$-equivariant local system $\aE_Y$ on
  $Y$.  The global nonabelian Hodge Correspondence for projective klt spaces
  then assigns a G-equivariant, locally free, polystable Higgs bundle
  $(ℰ_Y, θ_{Y}) = η_Y(γ^*(\aE_{X_{\reg}}))$ to $γ^*(\aE_{X_{\reg}})$, which is
  seen to coincide with
  $γ^*(ℰ_{X_{\reg}}, θ_{X_{\reg}}) = γ^*(η_{X_{\reg}}(\aE_{X_{\reg}}))$ wherever
  this makes sense, see Proposition~\ref{prop:p1} as well as Remark and
  Notation~\ref{rem:p1} and also again \cite[Prop.~3.10]{GKPT17}.  The fact that
  $η_Y$ is an equivalence of categories now quickly implies that $η_{X_{\reg}}$
  is full and faithful as well.
\end{proof}

\begin{rem}\label{rem:correspondences_coincide}
  It follows from the preceding proof and from Proposition~\ref{prop:p1} that on
  a maximally quasi-étale cover $Y$, the nonabelian Hodge correspondence for
  $\pHiggs_{Y_{\reg}}$ coincides with the nonabelian Hodge Correspondence for
  the projective klt space $Y$ when we apply the natural restriction functors on
  both sides of the correspondence.
\end{rem}

\subsection{Nonabelian Hodge correspondence for semistable bundles}
\approvals{Behrouz & yes \\ Daniel & yes \\ Stefan & yes\\ Thomas & yes}

In direct analogy to Simpson's work, Theorem~\ref{thm:klt-smooth1} extends to
give an equivalence between the category of flat bundles and arbitrary local
systems on $X_{\reg}$.  The (fairly standard) proof requires a version of the
nonabelian Hodge correspondence for a maximally quasi-étale cover, \cite[Thm.~3.4
and the discussion after Prop.~3.11]{GKPT17}, Theorem~\ref{thm:klt-smooth1}
above, and the formalities of differential graded categories (DGCs) established
in~\cite[Sect.~3]{MR1179076}.  The details are left to the reader.

\begin{factdef}[Category $\Higgs_{X_{\reg}}$]
  Given a projective klt space $X$ and a locally free Higgs sheaf
  $(ℰ_{X_{\reg}}, θ_{ℰ_{X_{\reg}}})$ on $X_{\reg}$, the following conditions are
  equivalent.
  \begin{enumerate}
  \item There exists an ample $H ∈ \Div(X)$, such that
    $(ℰ_{X_{\reg}}, θ_{ℰ_{X_{\reg}}})$ is semistable and has vanishing $ℚ$-Chern
    classes with respect to $H$.

  \item For any ample $H ∈ \Div(X)$, the Higgs bundle
    $(ℰ_{X_{\reg}}, θ_{ℰ_{X_{\reg}}})$ is semistable and has vanishing $ℚ$-Chern
    classes with respect to $H$.
  \end{enumerate}
  With their natural morphisms, the Higgs bundle satisfying these conditions
  form a category, which we denote by $\Higgs_{X_{\reg}}$.  \qed
\end{factdef}

In analogy to \cite[Cor.~3.10]{MR1179076}, the nonabelian Hodge correspondence
for semistable bundles now reads as follows.

\begin{thm}[Nonabelian Hodge correspondence for $\Higgs_{X_{\reg}}$]\label{thm:klt-smooth2}
  Let $X$ be a projective klt space.  Then, there exists equivalence between the
  category $\Higgs_{X_{\reg}}$ and the category $\LSys_{x_{\reg}}$ of local
  systems on $X_{\reg}$.

  For $(ℰ_{X_{\reg}}, θ_{ℰ_{X_{\reg}}}) ∈ \Higgs_{X_{\reg}}$ that are
  restrictions of bundles $(ℰ_X, θ_{ℰ_X}) ∈ \Higgs_X$ and for local systems
  $\aE_{X_{\reg}} ∈ \LSys_{X_{\reg}}$ that are restrictions of local systems
  $\aE_X ∈ \LSys_X$, the correspondence is compatible with the global nonabelian
  Hodge Correspondence for projective klt spaces found in
  \cite[Sect.~3]{GKPT17}.  \qed
\end{thm}

A statement similar to Remark~\ref{rem:correspondences_coincide} continues to
hold for the nonabelian Hodge correspondence for semistable bundles on a
maximally quasi-étale cover.

%
%
\svnid{$Id: 07-uniformisation.tex 1269 2019-02-27 07:58:09Z taji $}

\section{Proof of Theorem~\ref*{thm:ballrevisited}, uniformisation for minimal varieties}
\approvals{Behrouz & yes \\ Daniel & yes \\ Stefan & yes \\ Thomas & yes}
\label{sec:potbrv}

In this section, we prove Theorem~\ref{thm:ballrevisited}.  The strategy of
proof in principle follows \cite[Prop.~8.2 and 8.3]{GKPT15}.  The main new
difficulty stems from the fact that a general complete intersection surface in a
klt surface need not be smooth, but might have finite quotient singularities.
We maintain notation and assumptions of Theorem~\ref{thm:ballrevisited}
throughout.

\subsection*{Step 1: Reduction steps.}
\approvals{Behrouz & yes \\ Daniel & yes \\ Stefan & yes \\ Thomas & yes}

Let $π: X → X_{\can}$ be the birational crepant morphism to the canonical model
$X_{\can}$, which is also klt, and whose canonical divisor $K_{X_{\can}}$ is
ample, cf.~\cite[Thm.~3.3]{KM98}.

\begin{claim}\label{claim:7-1}
  We have the inequality
  $\what c_2(𝒯_{X_{\can}})·[K_{X_{\can}}]^{n-2} ≤ \what c_2(𝒯_X) · [K_X]^{n-2}$.
\end{claim}
\begin{proof}
  Let $S_{\can} ⊂ X_{\can}$ be a surface cut out by general members of the ample
  linear system $|m·K_{X_{\can}}|$ for $m ≫ 0$, and let $S ⊂ X$ be its preimage
  in $X$.  Notice that both $S_{\can}$ and $S$ have finite quotient klt
  singularities.  To show Claim~\ref{claim:7-1}, it is then equivalent to show
  that
  $$
  \what{c}_2(𝒯_{X_{\can}}|_{S_{\can}}) = \what{c}_2(𝒯_{X_{\can}})·[S_{\can}] ≤
  \what{c}_2(𝒯_X)·[S] = \what{c}_2(𝒯_X|_S).
  $$
  Now the exact sequences of $ℚ$-vector bundles
  $$
  0 → 𝒯_{S_{\can}} → 𝒯_{X_{\can}}|_{S_{\can}} → \sN_{S_{\can}/X_{\can}} → 0
  \quad\text{and}\quad 0 → 𝒯_S → 𝒯_X|_S → \sN_{S/X} → 0
  $$
  give equalities of $ℚ$-Chern numbers,
  $$
  \begin{matrix}
    \what{c}_2(𝒯_{X_{\can}}|_{S_{\can}}) & = & \what{c}_2(𝒯_{S_{\can}}) & + & \what{c}_1(𝒯_{S_{\can}}) · c_1(\sN_{S_{\can}/X_{\can}}) & + & \what{c}_2(\sN_{S_{\can}/X_{\can}}) \\
    \what{c}_2(𝒯_X|_S) & = & \what{c}_2(𝒯_S) & + & \what{c}_1(𝒯_S)·c_1(\sN_{S/X}) & + & \what{c}_2(\sN_{S/X}).
  \end{matrix}
  $$
  There is more that we can say about the summands on the right hand side.
  First, recall from \cite[Thm.~4.2]{Meg97} that
  $\what{c}_2(𝒯_{S_{\can}}) ≤ \what{c}_2( 𝒯_S)$.  Second, recalling that the
  morphism $S → S_{\can}$ is crepant, that $\sN_{S_{\can}/X_{\can}}$ is locally
  free and that $\sN_{S/X} ≅ (π|_S)^* \sN_{S_{\can}/X_{\can}}$, we find
  equalities,
  $$
  \what{c}_1(𝒯_{S_{\can}}) · c_1(\sN_{S_{\can}/X_{\can}}) = \what{c}_1(𝒯_S) ·
  c_1(\sN_{S(X}) \quad\text{and}\quad \what{c}_2(\sN_{S_{\can}/X_{\can}}) =
  \what{c}_2(\sN_{S/X}),
  $$
  for $ℚ$-Chern numbers on $S_{\can}$ and $S$, respectively.
  \qedhere~(Claim~\ref{claim:7-1})
\end{proof}

As a direct consequence of Claim~\ref{claim:7-1} and of the fact that $π$ is
crepant, we see that equality holds in the $ℚ$-Miyaoka-Yau inequality for
$X_{\can}$ as well.  The variety $X_{\can}$ therefore reproduces the assumptions
made in Theorem~\ref{thm:ballrevisited}, and we may assume for the remainder of
the present proof that the divisor $K_X$ is ample.

Likewise, if $γ: Y → X$ is any quasi-étale cover, recall from
\cite[Prop.~5.20]{KM98} that $Y$ is again klt.  We have remarked in
\cite[Lem.~3.16]{GKPT15} that equality holds in the $ℚ$-Miyaoka-Yau inequality
for $Y$, too.  Replacing $X$ by a suitable maximally quasi-étale cover,
\cite[Thm.~1.5]{GKP13}, we will therefore assume from now on that $X$ is
maximally quasi-étale.  Our aim is now to show that $X$ is smooth.  Once this is
established, the main claim will follow from classical uniformisation results of
Yau for smooth projective varieties.

\subsection*{Step 2: End of proof.}
\approvals{Behrouz & yes \\ Daniel & yes \\ Stefan & yes \\ Thomas & yes}

Consider the locally free sheaf $ℰ_{X_{\reg}} := Ω_{X_{\reg}}¹ ⊕ 𝒪_{X_{\reg}}$
on $X_{\reg}$ and its natural Higgs field
$$
θ_{ℰ_{X_{\reg}}} : Ω_{X_{\reg}}¹ ⊕ 𝒪_{X_{\reg}} → \bigl(Ω_{X_{\reg}}¹ ⊕
𝒪_{X_{\reg}}\bigr)⊗Ω_{X_{\reg}}¹, \qquad (a,b) ↦ (0,1) ⊗ a.
$$
Recall from \cite[Cor.~7.2]{GKPT15} that the Higgs sheaf
$(ℰ_{X_{\reg}}, θ_{ℰ_{X_{\reg}}})$ is stable with respect to the ample divisor
$K_X$.  Lemma~\ref{lem:stvbn} then asserts that the endomorphism bundle with its
natural Higgs field
$$
(ℱ_{X_{\reg}}, θ_{ℱ_{X_{\reg}}}) := \sEnd(ℰ_{X_{\reg}},\, θ_{ℰ_{X_{\reg}}})
$$
is polystable with respect to $K_X$.  Let $ℱ_X$ be the unique extension of
$ℱ_{X_{\reg}}$ as a reflexive sheaf on $X$.  By construction, $c_1(ℱ_X) = 0$.
As an immediate consequence of the assumed Equality~\eqref{eq:MY} and the
calculus of $ℚ$-Chern classes, see \cite[Lem.~3.18]{GKPT15}, we obtain the
vanishing $\what{ch}_2(ℱ_X)·[K_X]^{n-2} = \what{c}_2(ℱ_X)·[K_X]^{n-2} = 0$.  But
then Theorem~\ref{thm:Qflat1new} implies that
$$
(ℱ_{X_{\reg}}, θ_{ℱ_{X_{\reg}}}) ∈ \TPIH_{X_{\reg}}.
$$
Proposition~\ref{prop:p1} together with the assumption that $X$ is maximally
quasi-étale implies that $ℱ_X$ is locally free.  Now, as $𝒯_X$ is a direct
summand of $ℱ_X = \sEnd( Ω^{[1]}_X ⊕ 𝒪_X)$, it follows that $𝒯_X$ is locally
free.  The solution of the Lipman-Zariski conjecture for klt spaces,
\cite[Thm.~6.1]{GKKP11}, therefore asserts that $X$ is smooth.
Theorem~\ref{thm:ballrevisited} now follows from the classical result of Yau,
\cite[Rem.~(iii) on p.~1799]{MR0451180}.  \qed

%
%
\svnid{$Id: 08-hyperbolicity.tex 1264 2019-02-26 17:35:29Z kebekus $}

\section{Positivity in the sheaf of reflexive differentials}\label{sec:positivity}
\subversionInfo
\approvals{Behrouz & yes \\ Daniel & yes \\ Stefan & yes \\ Thomas & yes}

Given a singular ball quotient $X$, we ask for positivity in the sheaf
$Ω^{[1]}_X$ of reflexive differentials.  In other words, we would like to answer
the following question.

\begin{question}[Positivity for singular ball quotients]\label{question:1}
  Given a morphism $f: Y → X$ of projective varieties where $X$ has canonical or
  klt singularities and where $Y$ is a smooth ball quotient of general type, is
  $Ω^{[1]}_X$ positive in a suitable sense?
\end{question}

The relevance of Question~\ref{question:1} is illustrated by
Proposition~\ref{prop:nefandratlcurves} below, which relates positivity to
non-existence of rational curves.  Moreover, it is motivated by the
hyperbolicity statement \cite[Cor.~1.4 and Sect.~9.3]{GKPT15}.  The answer to
Question~\ref{question:1} turns out to be surprisingly delicate.  One the one
hand, we show in Section~\ref{ssec:p2} that sufficiently high symmetric powers
$\Sym^{[m]} Ω^{[1]}_X$ are always ample in the sense of
Definition~\ref{def:possheaf}.  On the other hand, Section~\ref{ssec:p3} shows
by way of example that even nefness of $Ω^{[1]}_X$ fails in general.  In this
sense, we have no satisfactory answer to Question~\ref{question:1} at present.

\subsection{Consequences of positivity}
\approvals{Behrouz & yes \\ Daniel & yes \\ Stefan & yes \\ Thomas & yes}

As pointed out above, positivity in the singular sheaf $Ω^{[1]}_X$ directly
relates to hyperbolicity properties of the underlying variety.

\begin{prop}\label{prop:nefandratlcurves}
  Let $X$ be a projective klt space.
  \begin{enumerate}
  \item\label{il:u1} If $Ω^{[1]}_X$ is nef, then $X$ does not contain rational
    curves.
  \item\label{il:u2} If $Ω^{[1]}_X$ is ample, then $X$ does not contain any
    curve whose normalisation is of genus one.
  \end{enumerate}
\end{prop}

The proof of Proposition~\ref{prop:nefandratlcurves} uses the fact that there
exists a functorial pull-back functor for reflexive differential forms on klt
spaces that agrees with the standard pull-back of Kähler differentials wherever
that makes sense.  We refer to \cite[Sect.~5]{MR3084424} for a precise
reference, and to \cite[Sect.~3]{GKT16} for an overview.

\begin{lem}[Pull-back of reflexive differential is generically surjective]\label{lem:pull}
  Let $X$ be a quasi-projective klt space and $Y ⊆ X$ be a smooth subvariety,
  with inclusion $ι: Y → X$.  Then, the pull-back map $dι: ι^*Ω^{[1]}_X → Ω¹_Y$
  is generically surjective.
\end{lem}
\begin{proof}
  The problem is local in the étale topology.  More precisely, since we aim to
  prove \emph{generic} surjectivity, we are free to replace $X$ by any
  Zariski-open subset that intersects $Y$.  Recalling that pull-back of
  reflexive differentials is functorial, \cite[Sect.~5.1]{MR3084424} or
  \cite[Sect.~14]{KS18}, we may also replace $X$ and $Y$ by finite étale covers.
  We will now project to $Y$.  To be precise, passing to an étale cover if
  necessary, \cite[Prop.~2.26]{GKKP11} allows us to assume without loss of
  generality that there exists a morphism $Φ : X → Y$ whose restriction to $Y$
  is isomorphic.  Surjectivity of $dι$ then comes out of the following diagram,
  $$
  \xymatrix{ %
    Ω¹_Y \ar[rr]_{dι} \ar@{->>}@/^4mm/[rrrr]^{d (Φ◦ι)} && ι^* Ω^{[1]}_X
    \ar[rr]_{ι^* dΦ} && Ω¹_Y, }
  $$
  whose commutativity is again due to the functoriality of reflexive pull-back.
\end{proof}

\begin{proof}[Proof of Proposition~\ref{prop:nefandratlcurves}]
  We prove \ref{il:u2} only.  Assume that $Ω^{[1]}_X$ is ample and let $C ⊂ X$
  be any irreducible curve with inclusion $ι: C → X$ and normalisation
  $η: \wtilde{C} → C$.  Write $f := ι◦η$ and notice that Lemma~\ref{lem:pull}
  implies that the pull-back map $df: f^{[*]} Ω^{[1]}_X → Ω¹_{\wtilde{C}}$ is
  generically surjective.  Using that $\wtilde{C}$ is smooth, observe that
  $$
  f^{[*]} Ω^{[1]}_X ≅ \factor{f^* Ω^{[1]}_X}{\tor},
  $$
  as both sheaves are locally free by \cite[Cor.~on p.~148]{OSS}.  Consequently,
  $f^{[*]} Ω^{[1]}_X$ is a quotient of the ample sheaf $f^* Ω^{[1]}_X$ and
  therefore itself ample.  We conclude that $\deg Ω¹_{\wtilde{C}} >0$, and hence
  $g(\wtilde C) > 1$, as claimed.
\end{proof}

\subsection{Positivity of symmetric differentials}\label{ssec:p2}
\approvals{Behrouz & yes \\ Daniel & yes \\ Stefan & yes \\ Thomas & yes}

As pointed out in the introduction, we show that the sheaves of reflexive
symmetric differentials of sufficiently high degree are always ample on singular
ball quotients.

\begin{prop}[Positivity of symmetric differentials for singular ball quotients]\label{prop:H1}
  Let $f: Y → X$ be a quasi-étale cover, where $X$ is projective with klt
  singularities.  If $Y$ is smooth and if $Ω¹_Y$ is ample\footnote{For instance,
    this is the case when $Y$ is a smooth ball quotient.}, then the following
  will hold.
  \begin{enumerate}
  \item\label{il:f1} The sheaf $f^{[*]} Ω^{[1]}_X$ is ample.
  \item\label{il:f2} The sheaf $\Sym^{[m]} Ω^{[1]}_X$ is ample for $m ≫ 0$
    sufficiently divisible.
  \end{enumerate}
\end{prop}
\begin{proof}
  To prove \ref{il:f1}, recall from \cite[Thm.~1.3]{MR3084424} or
  \cite[Thm.~14.1]{KS18} that there exists a pull-back morphism for reflexive
  differential forms, $df: f^{[*]} Ω^{[1]}_X → Ω¹_Y$, which is an isomorphism on
  the big open set where $f$ is étale.  Since both sheaves are reflexive, $df$
  is actually an isomorphism, and the first assertion follows.
  
  To prove \ref{il:f2}, recall from \cite[Thm.~3.1]{MR670921} that the sheaf
  $f_* \Sym^m Ω¹_Y$ is ample for all sufficiently large and divisible natural
  numbers $m$.  Over the smooth part of $X$, the projection formula gives
  $$
  f_* \Sym^m Ω¹_Y ≅ f_*f^* \Sym^m Ω¹_X ≅ \Sym^m Ω¹_X ⊕ \left( \Sym^m
    Ω¹_X ⊗ \factor{f_* 𝒪_Y}{𝒪_X} \right).
  $$
  Since $f_* \Sym^m Ω¹_Y$ is reflexive, this implies
  $$
  f_* \Sym^m Ω¹_Y ≅ \Sym^{[m]} Ω^{[1]}_X ⊕ \left( \Sym^m Ω¹_X ⊗ \factor{f_*
      𝒪_Y}{𝒪_X} \right)^{**}.
  $$
  It follows that $\Sym^{[m]} Ω^{[1]}_X$ is ample as a direct summand of an
  ample sheaf.
\end{proof}

\subsection{Failure of positivity in general}\label{ssec:p3}
\approvals{Behrouz & yes \\ Daniel & yes \\ Stefan & yes \\ Thomas & yes}

In spite of the positivity result established in Proposition~\ref{prop:H1},
Question~\ref{question:1} has a negative answer in general.  First examples
already exist in dimension two.

\begin{example}[A klt ball quotient whose reflexive cotangent sheaf is not ample]
  The example given in \cite[Sect.~9.4]{GKPT15} shows that there exists a klt
  ball quotient surface $S$ that is covered by curves whose normalisations are
  elliptic.  As a consequence of Proposition~\ref{prop:nefandratlcurves} above,
  $Ω^{[1]}_S$ is not ample.
\end{example}

\begin{example}[A canonical ball quotient whose reflexive cotangent sheaf is not nef]
  Recall from \cite{MR2586735} that there exists a fake projective plane $Y$
  that admits an automorphism $σ$ of order three.  Recall from
  \cite[p.~233]{MR527834} that $Y$ is a smooth ball quotient.  The quotient
  surface $X := Y/\langle σ \rangle$ has been studied by Keum.  He proves in
  \cite[Prop.~3.1]{MR2443971} that $X$ has exactly three singular points, which
  are canonical of type $A_2$, and that $K_X$ is Cartier with $[K_X]² = 3$.
  With this description of $X$ at hand, the following Proposition~\ref{prop:H2}
  shows that $Ω^{[1]}_X$ cannot possibly be nef.
\end{example}

\begin{prop}\label{prop:H2}
  Let $X$ be a projective surface with non-trivial canonical singularities where
  $Ω^{[1]}_X$ is nef.  If $\:[K_X]² = 3$, then $X$ has exactly one singular
  point.
\end{prop}
\begin{proof}
  Recall from \cite[Thm.~4.20]{KM98} or \cite[p.~347]{Reid87} that $X$ has
  \emph{Du Val surface singularities}.  Among other things, this implies that
  $K_X$ is Cartier, and that the minimal desingularisation $π: \wtilde{X} → X$
  is \emph{crepant}, $K_{\wtilde{X}} = π^* K_X$.  Denote the singular points of
  $X$ by $x_1$, …, $x_k$.  The exceptional set $E ⊂ \wtilde{X}$ therefore
  consists of a number $(-2)$-curves arranged in $k$ connected components,
  $$
  E = (E_{1,1} ∪ ⋯ ∪ E_{1,n_1}) \sqcup ⋯ \sqcup (E_{k,1} ∪ ⋯ ∪ E_{1,n_k}).
  $$
  We aim to analyse $Ω^{[1]}_X$ via its reflexive pull-back
  $ℰ := π^{[*]} Ω^{[1]}_X$, which is a locally free sheaf on $\wtilde{X}$,
  cf.~\cite[Chap.~II, Lem.~1.1.10]{OSS}.

  \begin{claim}\label{cl:p1}
    The sheaf $ℰ$ is nef.  In particular, $c_1(ℰ)² ≥ 0$.
  \end{claim}
  \begin{proof}[Proof of Claim~\ref{cl:p1}]
    Consider the natural map from $𝒜 := π^* Ω^{[1]}_X$ into its reflexive hull
    $ℰ$, which factors as follows
    $$
    \xymatrix{%
      𝒜 \ar@{->>}[r] & 𝒜/\tor \: \ar@{^(->}[r] & ℰ %
    }
    $$
    Nefness of $ℰ$ can now be concluded using Fact~\ref{fact:1}: as a pull-back
    of the nef sheaf $Ω^{[1]}_X$, we see that $𝒜$ nef and then so is its
    quotient $ℬ := 𝒜/\tor$.  Since $\wtilde{X}$ is a surface, the torsion free
    sheaf $ℬ$ is locally free outside of a finite set by \cite[Cor.~on
    p.~148]{OSS}, and therefore agrees with $ℰ$ away from that set.  If $C$ is
    any smooth curve and $γ: C → X$ any non-constant morphism, the natural map
    $γ^*ℬ → γ^*ℰ$ is thus generically injective.  Using that $C$ is a curve,
    decompose
    $$
    γ^*ℬ = \underbrace{(γ^*ℬ)/\tor}_{\text{locally free, nef}} ⊕ \tor(γ^*ℬ).
    $$
    The injection $(γ^*ℬ)/\tor ↪ γ^* ℰ$ then shows nefness.
    \qedhere~(Claim~\ref{cl:p1})
  \end{proof}
  
  The existence of $(-2)$-curves in $\wtilde{X}$ shows that $Ω¹_{\wtilde{X}}$ is
  not nef.  Claim~\ref{cl:p1} thus shows that the canonical pull-back morphism
  for reflexive differentials, $dπ: ℰ → Ω¹_{\wtilde{X}}$, which exists by
  \cite[Thm.~4.3]{GKKP11} or \cite[Thm.~1.9]{KS18}, is not isomorphic at the
  generic point of any such curve.  Somewhat more quantitatively, we find
  strictly positive $λ_{i,j} ∈ ℕ$ such that
  $$
  \textstyle \det ℰ ≅ ω_{\wtilde X}(-\sum_{i,j} λ_{i,j}·E_{i,j}) =
  ω_{\wtilde X}(-\sum_i F_i), \quad\text{for } F_i := \sum_j λ_{i,j}·E_{i,j}.
  $$
  Observe that $[F_i]·[F_j] =0$ for $i≠j$ and that $[F_i]² < 0$ for all $i$.
  Better still, since $[F_i]·[K_{\wtilde X}] = 0$, Riemann-Roch on $\wtilde X$
  asserts that the numbers $[F_i]²$ are all even.  But then
  $$
  c_1(ℰ)² = \bigl[ π^* K_X - \textstyle \sum_{i=1}^k F_i \bigr]² =
  \underbrace{[K_X]²}_{= 3} + \sum_{i=1}^k\underbrace{[F_i]²}_{\mathclap{≤
      -2\text{, even}}} \overset{\text{Claim~\ref{cl:p1}}}{≥} 0.
  $$
  We obtain that $k = 1$, as claimed.
\end{proof}


\begin{thebibliography}{GKKP11}

\bibitem[Anc82]{MR670921}
Vincenzo Ancona.
\newblock Faisceaux amples sur les espaces analytiques.
\newblock {\em Trans. Amer. Math. Soc.}, 274(1):89--100, 1982.
\newblock \href{http://dx.doi.org/10.2307/1999498}{DOI:10.2307/1999498}.

\bibitem[Bri98]{MR1451789}
Michel Brion.
\newblock Differential forms on quotients by reductive group actions.
\newblock {\em Proc. Amer. Math. Soc.}, 126(9):2535--2539, 1998.
\newblock
  \href{http://dx.doi.org/10.1090/S0002-9939-98-04320-2}{DOI:10.1090/S0002-9939-98-04320-2}.

\bibitem[BS95]{BS95}
Mauro~C. Beltrametti and Andrew~J. Sommese.
\newblock {\em The adjunction theory of complex projective varieties},
  volume~16 of {\em De Gruyter Expositions in Mathematics}.
\newblock Walter de Gruyter \& Co., Berlin, 1995.
\newblock
  \href{http://dx.doi.org/10.1515/9783110871746}{DOI:10.1515/9783110871746}.

\bibitem[CS10]{MR2586735}
Donald~I. Cartwright and Tim Steger.
\newblock Enumeration of the 50 fake projective planes.
\newblock {\em C. R. Math. Acad. Sci. Paris}, 348(1-2):11--13, 2010.
\newblock
  \href{http://dx.doi.org/10.1016/j.crma.2009.11.016}{DOI:10.1016/j.crma.2009.11.016}.

\bibitem[Ful98]{Fulton98}
William Fulton.
\newblock {\em Intersection {T}heory}, volume~2 of {\em Ergebnisse der
  Mathematik und ihrer Grenzgebiete. 3. Folge. A Series of Modern Surveys in
  Mathematics [Results in Mathematics and Related Areas. 3rd Series. A Series
  of Modern Surveys in Mathematics]}.
\newblock Springer-Verlag, Berlin, second edition, 1998.
\newblock
  \href{http://dx.doi.org/10.1007/978-1-4612-1700-8}{DOI:10.1007/978-1-4612-1700-8}.

\bibitem[GKKP11]{GKKP11}
Daniel Greb, Stefan Kebekus, Sándor~J. Kovács, and Thomas Peternell.
\newblock Differential forms on log canonical spaces.
\newblock {\em Inst. {H}autes {É}tudes Sci.~{P}ubl.~{M}ath.}, 114(1):87--169,
  November 2011.
\newblock
  \href{http://dx.doi.org/10.1007/s10240-011-0036-0}{DOI:10.1007/s10240-011-0036-0}
  An extended version with additional graphics is available as
  \href{http://arxiv.org/abs/1003.2913}{arXiv:1003.2913}.

\bibitem[GKP16a]{GKP15}
Daniel Greb, Stefan Kebekus, and Thomas Peternell.
\newblock Movable curves and semistable sheaves.
\newblock {\em Int. Math. Res. Not.}, 2016(2):536--570, 2016.
\newblock
  \href{http://dx.doi.org/10.1093/imrn/rnv126}{DOI:10.1093/imrn/rnv126}.
  Preprint \href{http://arxiv.org/abs/1408.4308}{arXiv:1408.4308}.

\bibitem[GKP16b]{GKP13}
Daniel Greb, Stefan Kebekus, and Thomas Peternell.
\newblock Étale fundamental groups of {K}awamata log terminal spaces, flat
  sheaves, and quotients of abelian varieties.
\newblock {\em Duke Math. J.}, 165(10):1965--2004, 2016.
\newblock
  \href{http://dx.doi.org/10.1215/00127094-3450859}{DOI:10.1215/00127094-3450859}.
  Preprint \href{http://arxiv.org/abs/1307.5718}{arXiv:1307.5718}.

\bibitem[GKPT15]{GKPT15}
Daniel Greb, Stefan Kebekus, Thomas Peternell, and Behrouz Taji.
\newblock The {M}iyaoka-{Y}au inequality and uniformisation of canonical
  models.
\newblock Preprint \href{http://arxiv.org/abs/1511.08822}{arXiv:1511.08822}. To
  appear in the Annales scientifiques de l'École normale supérieure, November
  2015.

\bibitem[GKPT19]{GKPT17}
Daniel Greb, Stefan Kebekus, Thomas Peternell, and Behrouz Taji.
\newblock Nonabelian {H}odge theory for klt spaces and descent theorems for
  vector bundles.
\newblock {\em Compos. Math.}, 155(2):289--323, 2019.
\newblock
  \href{https://doi.org/10.1112/S0010437X18007923}{DOI:10.1112/S0010437X18007923}.
  Preprint \href{https://arxiv.org/abs/1711.08159}{arXiv:1711.08159}.

\bibitem[GKT18]{GKT16}
Daniel Greb, Stefan Kebekus, and Behrouz Taji.
\newblock Uniformisation of higher-dimensional varieties.
\newblock In Tommaso de~Fernex, Brendan Hassett, Mircea Mustaţă, Martin
  Olsson, Mihnea Popa, and Richard Thomas, editors, {\em Algebraic Geometry:
  Salt Lake City 2015}, volume~1 of {\em Proceedings of Symposia in Pure
  Mathematics}, pages 277--308. American Mathematical Society, American
  Mathematical Society, 2018.
\newblock Preprint \href{http://arxiv.org/abs/1608.06644}{arXiv:1608.06644}.

\bibitem[GM88]{GoreskyMacPherson}
Mark Goresky and Robert~D. MacPherson.
\newblock {\em Stratified {M}orse theory}, volume~14 of {\em Ergebnisse der
  Mathematik und ihrer Grenzgebiete (3) [Results in Mathematics and Related
  Areas (3)]}.
\newblock Springer-Verlag, Berlin, 1988.
\newblock
  \href{http://dx.doi.org/10.1007/978-3-642-71714-7}{DOI:10.1007/978-3-642-71714-7}.

\bibitem[Gro61]{GrCa60}
Alexandre Grothendieck.
\newblock Techniques de construction en géométrie analytique. {V}. {F}ibrés
  vectoriels, fibrés projectifs, fibrés en drapeaux.
\newblock In {\em Séminaire Henri Cartan}, volume~13, chapter Exp.\ n° 12,
  pages 1--15. Secrétariat mathématique, Paris, 1960--1961.
\newblock Revised in collaboration with Jean Dieudonné.
  \href{http://www.numdam.org/item/SHC_1960-1961__13_1_A8_0}{numdam.SHC-1960-1961-13-1-A8-0}.

\bibitem[Gro70]{MR0262386}
Alexander Grothendieck.
\newblock Représentations linéaires et compactification profinie des groupes
  discrets.
\newblock {\em Manuscripta Math.}, 2:375--396, 1970.
\newblock \href{http://dx.doi.org/10.1007/BF01719593}{DOI:10.1007/BF01719593}.

\bibitem[GT16]{GT16}
Henri Guenancia and Behrouz Taji.
\newblock Orbifold stability and {M}iyaoka-{Y}au inequality for minimal pairs.
\newblock Preprint \href{http://arxiv.org/abs/1611.05981}{arXiv:1611.05981},
  November 2016.

\bibitem[GT17]{GT13}
Daniel Greb and Matei Toma.
\newblock Compact moduli spaces for slope-semistable sheaves.
\newblock {\em Algebr. Geom.}, 4(1):40--78, 2017.
\newblock
  \href{http://dx.doi.org/10.14231/AG-2017-003}{DOI:10.14231/AG-2017-003}.
  Preprint \href{http://arxiv.org/abs/1303.2480}{arXiv:1303.2480}.

\bibitem[Har77]{Ha77}
Robin Hartshorne.
\newblock {\em Algebraic geometry}.
\newblock Springer-Verlag, New York, 1977.
\newblock Graduate Texts in Mathematics, No. 52.
  \href{http://dx.doi.org/10.1007/978-1-4757-3849-0}{DOI:10.1007/978-1-4757-3849-0}.

\bibitem[HL10]{MR2665168}
Daniel Huybrechts and Manfred Lehn.
\newblock {\em The geometry of moduli spaces of sheaves}.
\newblock Cambridge Mathematical Library. Cambridge University Press,
  Cambridge, second edition, 2010.
\newblock
  \href{http://dx.doi.org/10.1017/CBO9780511711985}{DOI:10.1017/CBO9780511711985}.

\bibitem[JZ97]{JZ97}
Jürgen Jost and Kang Zuo.
\newblock Harmonic maps of infinite energy and rigidity results for
  representations of fundamental groups of quasiprojective varieties.
\newblock {\em Journal of Differential Geometry}, 47:469--503, 1997.
\newblock
  \href{http://projecteuclid.org/euclid.jdg/1214460547}{euclid.jdg/1214460547}.

\bibitem[Keb13]{MR3084424}
Stefan Kebekus.
\newblock Pull-back morphisms for reflexive differential forms.
\newblock {\em Adv. Math.}, 245:78--112, 2013.
\newblock
  \href{http://dx.doi.org/10.1016/j.aim.2013.06.013}{DOI:10.1016/j.aim.2013.06.013}.
  Preprint \href{http://arxiv.org/abs/1210.3255}{arXiv:1210.3255}.

\bibitem[Keu08]{MR2443971}
Jong-Hae Keum.
\newblock Quotients of fake projective planes.
\newblock {\em Geom. Topol.}, 12(4):2497--2515, 2008.
\newblock \href{http://dx.doi.org/10.2140/gt.2008.12.2497}{DOI:
  10.2140/gt.2008.12.2497}.

\bibitem[KKV89]{MR1044586}
Friedrich Knop, Hanspeter Kraft, and Thierry Vust.
\newblock The {P}icard group of a {$G$}-variety.
\newblock In {\em Algebraische {T}ransformationsgruppen und
  {I}nvariantentheorie}, volume~13 of {\em DMV Sem.}, pages 77--87.
  Birkhäuser, Basel, 1989.

\bibitem[KM98]{KM98}
János Kollár and Shigefumi Mori.
\newblock {\em Birational geometry of algebraic varieties}, volume 134 of {\em
  Cambridge Tracts in Mathematics}.
\newblock Cambridge University Press, Cambridge, 1998.
\newblock With the collaboration of C.\ H.\ Clemens and A.\ Corti, Translated
  from the 1998 Japanese original.
  \href{http://dx.doi.org/10.1017/CBO9780511662560}{DOI:10.1017/CBO9780511662560}.

\bibitem[KS18]{KS18}
Stefan Kebekus and Christian Schnell.
\newblock Extending holomorphic forms from the regular locus of a complex space
  to a resolution of singularities.
\newblock Preprint \href{https://arxiv.org/abs/1811.03644}{arXiv:1811.03644},
  November 2018.

\bibitem[LP91]{MR1157844}
Joseph Le~Potier.
\newblock Fibrés de {H}iggs et systèmes locaux.
\newblock {\em Astérisque}, (201-203):Exp.\ No.\ 737, 221--268 (1992), 1991.
\newblock Séminaire Bourbaki, Vol. 1990/91.

\bibitem[LT14]{LT14}
Steven~S.Y. Lu and Behrouz Taji.
\newblock A characterization of finite quotients of {A}belian varieties.
\newblock {\em IMRN}, 2018(1):292--319, October 2014.
\newblock
  \href{http://dx.doi.org/10.1093/imrn/rnw251}{DOI:10.1093/imrn/rnw251}.
  Preprint \href{http://arxiv.org/abs/1410.0063}{arXiv:1410.0063}.

\bibitem[Meg99]{Meg97}
Gabor Megyesi.
\newblock Generalisation of the {B}ogomolov-{M}iyaoka-{Y}au inequality to
  singular surfaces.
\newblock {\em Proc. London Math. Soc. (3)}, 78(2):241--282, 1999.
\newblock
  \href{http://dx.doi.org/10.1112/S0024611599001719}{DOI:10.1112/S0024611599001719}.

\bibitem[Moc06]{MR2310103}
Takuro Mochizuki.
\newblock Kobayashi-{H}itchin correspondence for tame harmonic bundles and an
  application.
\newblock {\em Astérisque}, (309):viii+117, 2006.

\bibitem[Moc07a]{MR2281877}
Takuro Mochizuki.
\newblock Asymptotic behaviour of tame harmonic bundles and an application to
  pure twistor {$D$}-modules. {I}.
\newblock {\em Mem. Amer. Math. Soc.}, 185(869):xii+324, 2007.
\newblock \href{https://doi.org/10.1090/memo/0869}{DOI:10.1090/memo/0869}.

\bibitem[Moc07b]{MR2283665}
Takuro Mochizuki.
\newblock Asymptotic behaviour of tame harmonic bundles and an application to
  pure twistor {$D$}-modules. {II}.
\newblock {\em Mem. Amer. Math. Soc.}, 185(870):xii+565, 2007.
\newblock \href{https://doi.org/10.1090/memo/0870}{DOI:10.1090/memo/0870}.

\bibitem[MR84]{MR84}
Vikram~B. Mehta and Annamalai Ramanathan.
\newblock Restriction of stable sheaves and representations of the fundamental
  group.
\newblock {\em Invent. Math.}, 77(1):163--172, 1984.
\newblock \href{http://dx.doi.org/10.1007/BF01389140}{DOI: 10.1007/BF01389140}.

\bibitem[Mum79]{MR527834}
David Mumford.
\newblock An algebraic surface with {$K$}\ ample, {$(K^{2})=9$}, {$p_{g}=q=0$}.
\newblock {\em Amer. J. Math.}, 101(1):233--244, 1979.
\newblock \href{http://dx.doi.org/10.2307/2373947}{DOI:10.2307/2373947}.

\bibitem[Mum83]{MR717614}
David Mumford.
\newblock Towards an enumerative geometry of the moduli space of curves.
\newblock In {\em Arithmetic and geometry, {V}ol. {II}}, volume~36 of {\em
  Progr. Math.}, pages 271--328. Birkhäuser Boston, Boston, MA, 1983.
\newblock \href{http://dx.doi.org/10.1007/978-1-4757-9286-7_12}{DOI:
  10.1007/987-1-4757-9286-7\-12}.

\bibitem[OSS80]{OSS}
Christian Okonek, Michael Schneider, and Heinz Spindler.
\newblock {\em Vector bundles on complex projective spaces}, volume~3 of {\em
  Progress in Mathematics}.
\newblock Birkhäuser Boston, Mass., 1980.

\bibitem[Rei87]{Reid87}
Miles Reid.
\newblock Young person's guide to canonical singularities.
\newblock In {\em Algebraic geometry, Bowdoin, 1985 (Brunswick, Maine, 1985)},
  volume~46 of {\em Proc. Sympos. Pure Math.}, pages 345--414. Amer. Math.
  Soc., Providence, RI, 1987.

\bibitem[Sab13]{MR3087348}
Claude Sabbah.
\newblock Théorie de {H}odge et correspondance de {H}itchin-{K}obayashi
  sauvages (d'après {T}. {M}ochizuki).
\newblock {\em Astérisque}, (352):Exp. No. 1050, viii, 205--241, 2013.
\newblock Séminaire Bourbaki. Vol. 2011/2012. Exposés 1043--1058.

\bibitem[SBW94]{SBW94}
Nicholas~I. Shepherd-Barron and Pelham~M.H. Wilson.
\newblock Singular threefolds with numerically trivial first and second {C}hern
  classes.
\newblock {\em J. Algebraic Geom.}, 3(2):265--281, 1994.

\bibitem[Ser66]{MR0212214}
Jean-Pierre Serre.
\newblock Prolongement de faisceaux analytiques cohérents.
\newblock {\em Ann. Inst. Fourier (Grenoble)}, 16(fasc. 1):363--374, 1966.

\bibitem[Sim88]{MR944577}
Carlos~T. Simpson.
\newblock Constructing variations of {H}odge structure using {Y}ang-{M}ills
  theory and applications to uniformization.
\newblock {\em J. Amer. Math. Soc.}, 1(4):867--918, 1988.
\newblock \href{http://dx.doi.org/10.1090/S0894-0347-1988-0944577-9}{DOI:
  10.1090/S0894-0347-1988-0944577-9}.

\bibitem[Sim92]{MR1179076}
Carlos~T. Simpson.
\newblock Higgs bundles and local systems.
\newblock {\em Inst. Hautes Études Sci. Publ. Math.}, 75:5--95, 1992.
\newblock \href{http://dx.doi.org/10.1007/BF02699491}{DOI:10.1007/BF02699491},
  \href{http://www.numdam.org/item?id=PMIHES_1992__75__5_0}{numdam.PMIHES-1992-75-5-0}.

\bibitem[Yau77]{MR0451180}
Shing-Tung Yau.
\newblock Calabi's conjecture and some new results in algebraic geometry.
\newblock {\em Proc. Nat. Acad. Sci. U.S.A.}, 74(5):1798--1799, 1977.

\end{thebibliography}
\end{document}